%% file: main.tex
\documentclass[11pt]{article}
\usepackage[margin =1in]{geometry}
\usepackage{amssymb,amsmath,amsthm}
\usepackage{enumerate}
\usepackage{hyperref}
\usepackage{cite}
\usepackage[capitalize]{cleveref}
\usepackage{enumitem}
\usepackage{color}
\usepackage{xcolor}

\usepackage{cancel}
\usepackage{pgf}
\usepackage{svg}
\usepackage{pgfplots}
\usepackage{import}

\usepackage[utf8]{inputenc} 
\usepackage[T1]{fontenc}    
\usepackage{hyperref}       
\usepackage{url}            
\usepackage{booktabs}       
\usepackage{amsfonts}       
\usepackage{nicefrac}       
\usepackage{microtype}
\usepackage{multirow}
\usepackage{xfrac}
\usepackage{todonotes}
\usepackage{titlesec}
\usepackage{caption}

\usepackage{multirow}

\usepackage{caption}
\usepackage{float}

\titleformat{\subsubsection}[runin]
{\normalfont\normalsize\bfseries}{\thesubsubsection}{1em}{}

\usepackage{graphicx}

\usepackage{appendix}
\numberwithin{equation}{section}

\newcommand{\inclu}[0] {\ar@{^{(}->}}

\newcommand{\R}{\mathbb{R}}

\newcommand{\argmin}{\operatornamewithlimits{argmin}}

\newtheorem{thm}{Theorem}[section]
\newtheorem{theorem}[thm]{Theorem}

\newtheorem{proposition}[thm]{Proposition}

\newtheorem{lemma}[thm]{Lemma}

\newtheorem{corollary}[thm]{Corollary}

\newtheorem{assumption}[thm]{Assumption}

\newtheorem{remark}[thm]{Remark}

\crefname{claim}{claim}{claims}
\Crefname{claim}{Claim}{Claims}
\crefname{lem}{lemma}{lemmas}
\Crefname{lem}{Lemma}{Lemmas}
\crefname{algorithm}{algorithm}{algorithms}
\Crefname{algorithm}{Algorithm}{Algorithms}

\theoremstyle{definition}

\theoremstyle{definition}

\theoremstyle{definition}
\newtheorem*{claim*}{Claim}

\usepackage{mathtools}

\usepackage[ruled]{algorithm2e}

\usepackage[T1]{fontenc}
\usepackage{lmodern}
\crefname{figure}{Figure}{Figures}

\pgfplotsset{compat=1.18}

\begin{document}

    \title{A Parameter-Free Restart Scheme with Only\\ a Parallelizable $\log\log(1/\epsilon)$ Overhead}

    \author{Yue Wu\footnote{Johns Hopkins University, Department of Applied Mathematics and Statistics, \url{ywu166@jhu.edu}} \qquad Benjamin Grimmer\footnote{Johns Hopkins University, Department of Applied Mathematics and Statistics, \url{grimmer@jhu.edu}}}

	\date{}
	\maketitle

	\begin{abstract}

\input{abstract}
	\end{abstract}

    \input{intro}
    \input{prelim}

\input{theory}

    \paragraph{Acknowledgments.} This work was supported in part by the Air Force Office of Scientific Research under award number FA9550-23-1-0531. Benjamin Grimmer was additionally supported as a fellow of the Alfred P. Sloan Foundation.
    
    {\small
    \bibliographystyle{unsrt}
    \bibliography{references}
    }
    \appendix
    \input{appendix}
\end{document}

%% file: abstract.tex
It is well-known that first-order methods can offer accelerated convergence rates in the presence of growth structures. Restarting schemes provide a general tool for such speed-ups. These schemes typically either require unrealistic problem knowledge, incur logarithmic overhead factors in oracle complexity, and/or have a nontrivial initial burn-in phase. We present a parameter-free approach for restarting any first-order method, avoiding these three drawbacks. Our approach dynamically deploys parallel instances of a given first-order method communicating progress in the style of Renegar and Grimmer~\cite{Renegar2022restart}. Our optimized scheme avoids expensive burn-ins and only requires $\mathcal{O}(\log\log(1/\epsilon))$ parallel processes when the accelerated rate is sublinear.

%% file: intro.tex
\section{Introduction} \label{Sect:Intro}

In the design of first-order methods for convex problems of the form
\begin{equation} \label{eq:main-problem}
    f_\star = \min_{x\in\R^d} f(x),
\end{equation}
``restarting'' has become a standard technique.
Restarting is an approach to first-order optimization method design in which one periodically restarts a given algorithm, reinitializing at the current iterate while resetting any stepsize schedules or accumulated memory and momentum. 
Despite its simplicity, this approach has a long track record~\cite{nesterov2013,Odonoghue2015,Su2016,Roulet2020,Renegar2022restart} of offering provable theoretical or practical gains for algorithm convergence. This idea was discussed as early as~\cite{nemirovskii1985} simply as a technique for convergence analysis. Their choice of exactly when to restart the underlying algorithm relied, however, on (often unrealistic) knowledge of constants $\mu>0,p\geq 1$ such that the following growth bound holds:
\begin{equation}\label{eq:polynomial-growth}
    f(x) - f_\star \geq \mu\, \operatorname{dist}(x, X_\star)^p \qquad \forall x\in\{z : f(z)\leq f(x^{(0)})\},
\end{equation}
where $X_\star = \argmin_{x\in \R^d} f(x) \neq \emptyset$ is the nonempty set of minimizers, and $x^{(0)}$ is an initialization. The case $p=2$ recovers quadratic growth, a weaker condition than strong convexity.

Such growth bounds are widespread, holding for a general class of semialgebraic functions~\cite{bolte2017}. As examples, the works~\cite{nesterov2013,Giselsson2014,Odonoghue2015,fercoq2016,Wen2017,Liu2017,Fercoq2019,sujanani2025} showed speed-ups for smooth optimization via restarting accelerated momentum-type methods and~\cite{Yang2018,Roulet2020,Renegar2022restart,davis2024} for nonsmooth optimization with subgradient and smoothing-type methods. Among the many modern usages of restarting, first-order methods for conic optimization have had substantial practical success~\cite{Applegate2021PDLP1,xiong2024,applegate2026pdlp2}.

A proposed restarting procedure for a given first-order method must, at its core, provide a rule or scheme for when and how to restart the method. Central to existing restarting methods is determining a restart time for which one can ensure a contraction in the objective value gap (or a similar quantity like the distance to a minimizer) occurs. In designing such rules, existing methods fail to attain at least one of the following desirable properties:
\begin{itemize}
    \item[(A)] {\bf Problem Constant-Free Optimization.} Given additional knowledge of problem constants, one can often guarantee that it suffices to restart at a certain iteration count. For example, in $L$-smooth convex minimization with $p>2$ growth~\eqref{eq:polynomial-growth}, one can restart an accelerated method after $\mathcal{O}\left(\frac{\sqrt{L}}{ \mu^{\frac{1}{p}}} \cdot \frac{1}{(f(x^{(0)})-f_\star)^{\frac{1}{2}-\frac{1}{p}}}\right) $ iterations, ensuring a contraction in the objective value gap. Alternatively, if one knew $f_\star$, one could directly monitor the objective gap and restart once a prescribed contraction occurs. Although theoretically convenient, restarting approaches reliant on exact knowledge of any of $\mu,p,f_\star$ or other problem constants are often prohibitive in practice. Approaches of this type have been taken by~\cite{Gu2013,zhang2015,freund2018,Yang2018,Necoara2019}.
    \item[(B)] {\bf No Non-asymptotic Burn-In Time.} Some restart schemes are only guaranteed to work in a neighborhood of the minimizer~\cite{Gu2013,Wen2017}. As a result, there is a potentially long burn-in time that algorithms incur before the fast convergence due to restarting can begin. For example, the parameter-free parallel restarting scheme of~\cite{Renegar2022restart} when applied to $L$-smooth convex problems with growth of order $p>2$~\eqref{eq:polynomial-growth}, using $m=\lceil\log(1/\epsilon)\rceil$ processes, requires
    $$ \mathcal{O}\left(\sqrt{L \operatorname{dist}(x^{(0)}, X_\star)^2} + \frac{\sqrt{L}}{ \mu^{\frac{1}{p}}} \cdot \frac{1}{\epsilon^{\frac{1}{2}-\frac{1}{p}}} \right) $$
    iterations to reach an $\epsilon$-minimizer. Although the second term is the optimal order of complexity, the first term is non-negligible for some regimes of $\epsilon,L,\operatorname{dist}(x^{(0)}, X_\star)$.
    \item[(C)] {\bf Avoidance of Logarithmic Oracle Complexity Overheads.} Continuing the example of the above-mentioned parallel restarting scheme of~\cite{Renegar2022restart}, they require $\log(1/\epsilon)$ (parallelizable) gradient evaluations per iteration, leading to a log factor gap from the optimal order of complexity. As a result, they derive a 
    $$ \mathcal{O}\left(\sqrt{L \operatorname{dist}(x^{(0)}, X_\star)^2}\log(1/\epsilon) + \frac{\sqrt{L}}{ \mu^{\frac{1}{p}}} \cdot \frac{1}{\epsilon^{\frac{1}{2}-\frac{1}{p}}} \log(1/\epsilon)\right)$$
    gradient oracle complexity to reach an $\epsilon$-minimizer. Alternative schemes perform exponentially spaced grid searches over unknown parameters to approximately leverage approaches that require strong knowledge of problem constants~\cite{nesterov2013,Lin2015,Liu2017,Roulet2020,adcock2025,sujanani2025}. This too incurs logarithmic factors in oracle complexities.
\end{itemize}

Existing approaches to restart schemes all fall short on at least one of these metrics. 
As discussed above, approaches like those of~\cite{Gu2013,zhang2015,Necoara2019}, reliant on knowledge of problem structural parameters, and those of~\cite{freund2018}, reliant on knowledge of $f_\star$, fail (A). The techniques proposed by~\cite{Fercoq2019,ito2021,Renegar2022restart} incur non-asymptotic burn-in costs, failing (B). The restarting schemes of~\cite{Yang2018,Renegar2022restart} and parameter searching approaches~\cite{nesterov2013,Lin2015,Liu2017,Roulet2020,adcock2025,sujanani2025} suffer additional logarithmic factors, falling short of (C). Our main contribution is the design of a restarting scheme mitigating these three shortcomings. Namely, our proposed scheme will provide optimal iteration complexities while having the properties:
\begin{itemize}
    \item[(A)] When applied in the context of~\eqref{eq:polynomial-growth}, our scheme only requires a lower bound on the growth exponent $p$, which may be conservative without harming our non-asymptotic guarantees.
    \item[(B)] Under a mild monotonicity condition (satisfied by accelerated methods in both smooth and nonsmooth optimization), the burn-in time of our approach will always be less than the optimal order of complexity, preventing any non-asymptotic inefficiencies.
    \item[(C)] Across a range of smooth and nonsmooth problems, our scheme will incur at most a (parallelizable) $\log\log(1/\epsilon)$ factor in oracle complexity, improving on prior work's logarithmic factors.
\end{itemize}

In this work, we consider a general model of deterministic first-order methods and growth bounds to leverage via restarting. We consider a generic first-order method, denoted $\mathtt{fom}$, which, given an initialization $x^{(0)}$ and a target accuracy $\epsilon>0$, generates iterates $x^{(1)}, x^{(2)},\dots$ sequentially. For smooth convex problems, one could consider gradient descent or an accelerated, momentum-type method~\cite{Nesterov1983,Beck2009fista}. For nonsmooth convex problems, one could consider a subgradient method or a smoothing method~\cite{Nesterov2005smoothing,Beck2012smoothing}. The only requirement placed on this iterative method is that it possesses a convergence guarantee of the following type, providing a bound on the number of iterations needed to decrease the objective value by $\epsilon$: for any $\delta \ge 2\epsilon > 0$ and $D >0$, assume
\begin{equation}
\label{eq:intro-decrement-for-fom}
\begin{cases}
f(x^{(0)}) - f_\star \ge \delta \\
\operatorname{dist}(x^{(0)}, X_\star) \le D
\end{cases}
\implies \min\{ f(x^{(k)}): 0\le k\le K_{\mathtt{dec}}(\epsilon, \delta, D) \} \le f(x^{(0)}) -\epsilon.
\end{equation}
For example, if $f$ is an $L$-smooth convex function and $\mathtt{fom}$ is Nesterov's accelerated method, one has $K_{\mathtt{dec}}(\epsilon, \delta, D) = \sqrt{2LD^2/(\delta-\epsilon)}$, see Proposition~\ref{prop:accel-method}.

To generalize the polynomial growth bound~\eqref{eq:polynomial-growth}, we assume there exists a strictly increasing and invertible function $G:[0,+\infty) \to [0,+\infty)$ such that the following holds
\begin{equation}
\label{eq:intro-growth-structure}
f(x)-f_\star \ge G\left( \operatorname{dist}(x, X_\star) \right) \qquad \forall x\in\{z : f(z)\leq f(x^{(0)})\}.
\end{equation}
The typical growth bound~\eqref{eq:polynomial-growth} then corresponds to $G(t)=\mu\cdot t^p$.

\paragraph{Our Contributions.} We develop a parameter-free, parallel restarting scheme, addressing the limitations (A), (B), and (C). Our scheme follows the parallel structure of Renegar and Grimmer~\cite{Renegar2022restart} while avoiding the logarithmic overhead and the burn-in costs inherent to their strategy and parameter choices. Applying our approach yields optimal iteration complexities while only incurring a $\log\log(1/\epsilon)$ factor overhead in oracle complexity across the range of typical smooth and nonsmooth optimization settings, summarized in Section~\ref{Sect:cor-monotone}.

In particular, for any target accuracy $\epsilon>0$, our proposed restart scheme is parameterized by a strictly increasing and unbounded sequence $\epsilon_k>0$ with $\epsilon_0 = \epsilon/2$. Given such a sequence, our Corollary~\ref{cor:monotone} establishes an oracle complexity bound of
\begin{equation*}
    (\hat{m}+1) \sum_{k=0}^{\hat{m}} \left(1+\frac{2\epsilon_{k+1}}{\epsilon_k}\right)K_{\mathtt{dec}}(\epsilon_k, 2\epsilon_k, G^{-1}(2\epsilon_k)),\qquad \hat{m} = \inf\left\{k : \epsilon_k > f(x^{(0)}) - f_\star\right\}
\end{equation*}
for computing an $\epsilon$-minimizer.
Optimizing the selection of the sequence $\epsilon_k$ improves our scheme's convergence guarantees. As an example, consider minimizing a convex $f$ that is $L$-smooth and has $p\geq 4$-order growth~\eqref{eq:polynomial-growth} with $\mathtt{fom}$ as Nesterov's accelerated method. By selecting the doubly exponential sequence
$$\epsilon_k = \frac{\epsilon}{2 e}\exp\left((9/8)^k\right),$$
our restarting method provides an oracle complexity of
$$ \mathcal{O}\left( \frac{\sqrt{L}}{ \mu^{\frac{1}{p}}} \cdot \frac{1}{\epsilon^{\frac{1}{2}-\frac{1}{p}}} \cdot \log\log\left( \frac{f(x^{(0)}) - f_\star}{\epsilon} \right) \right) \ .  $$
This is the optimal oracle complexity~\cite{nemirovskii1985}, up to the above doubly logarithmic factor. Prior adaptive works have required at least a (singly) logarithmic factor overhead in this special case. This overhead factor can be fully parallelized, meaning that one can compute these $\log\log\left( \frac{f(x^{(0)}) - f_\star}{\epsilon}\right)$ oracle calls in parallel at each iteration. Hence, our scheme's iteration complexity is exactly the optimal order.

A $\log\log(1/\epsilon)$ overhead factor has previously arisen in the design of adaptive algorithms for strongly convex nonsmooth minimization: the restarting algorithm in~\cite{Carmon2022} incurs an additional $\mathcal{O}(\log\log(1/\epsilon))$ factor compared with the known optimal non-adaptive rate. Applying our restarting scheme to this specialized setting recovers their results. Hence, our scheme provides a generalization of this previously isolated $\log\log(1/\epsilon)$ phenomenon. At most an overhead of $\mathcal{O}(\log\log(1/\epsilon))$ is needed to adapt in settings seeking an improved sublinear convergence rate by restarting a first-order method with a growth bound~\eqref{eq:polynomial-growth} with unknown $\mu,p$. 

One may consider the improvement small from a single logarithmic factor to a doubly logarithmic factor, as it will often only improve performance by small factors. However, the considered settings are fundamental, and so our results narrow down the fundamental costs associated with adaptive algorithm design.
For stochastic problems seeking high probability guarantees, the lower bounds of~\cite{Carmon2025} showed a $\Omega(\log\log(1/\epsilon))$ factor cost is inherent to adapting to unknown parameters. Hence, one may conjecture that the doubly logarithmic cost of our scheme is optimal. We establish tightness of our analysis technique, showing it cannot improve beyond this doubly logarithmic limit. We leave open this important question of general lower bounds.

\paragraph{Outline.} Section 2 presents our dynamic restarting scheme and its related details. Section 3 establishes a unified theoretical guarantee and a useful corollary. Section 4 applies the unified theory to three general problem classes and derives parameter-free corollaries under growth order $p\ge 2$ in~\eqref{eq:polynomial-growth}. Section 5 treats the complementary case $p\in [1,2)$.

%% file: prelim.tex
\section{Preliminaries} \label{Sect:Prelim}

We briefly summarize our notation. We use $\mathcal{O}(\cdot)$, $\Omega(\cdot)$, and $\Theta(\cdot)$ to denote the standard asymptotic notations, all with respect to the limit $\epsilon\to 0$. For any $y\in \mathbb{R}^d$, let $\operatorname{dist}(y, X_\star) = \inf_{x\in X_\star} \|y-x\|_2$, where $\|\cdot\|_2$ is the two-norm. For a convex function $f: \mathbb{R}^d \to \mathbb{R}\cup \{+\infty\}$, a vector $g \in \mathbb{R}^d$ is a subgradient of $f$ at $x_0 \in \mathbb{R}^d$ if $f(x) \ge f(x_0) + \langle g, x-x_0 \rangle$ for all $x \in \mathbb{R}^d$. The subdifferential of $f$ at $x_0$, defined as the set of all subgradients of $f$ at $x_0$, is denoted by $\partial f(x_0)$.

For notational ease, we define the maximum or the summation over an empty set to be zero. This convention allows us to avoid additional boundary cases in the iteration count bounds of Section~\ref{Sect:main-results}. For example, $\max_{K_1\le k\le K_2} h(k)=0$ and $\sum_{k=K_1}^{K_2} h(k) = 0$ if $K_1 > K_2$.

\subsection{Standard Assumptions}

{\bf Decrement guarantee.} As previously introduced in~\eqref{eq:intro-decrement-for-fom}, we consider any iterative first-order method $\mathtt{fom}(\epsilon)$ satisfying the following decrement guarantee. We require that the function value decreases by at least $\epsilon$ within $K_{\mathtt{dec}}(\epsilon, \delta, D)$ iterations, provided the initial objective gap is at least $\delta$ and the initial distance to $X_\star$ is at most $D$. In the following assumption, let $\{x^{(k)}\}_{k\ge 0}$ denote the sequence of iterates generated by $\mathtt{fom}(\epsilon)$ initialized at $x^{(0)}$.
\begin{assumption}[Decrement guarantee on $\mathtt{fom}$]
\label{assumption:decrement-for-fom}
For a fixed $\mathtt{fom}$, there exists a non-negative finite-valued function $K_{\mathtt{dec}}(\cdot,\cdot,\cdot)$ such that for any $\delta \ge 2\epsilon > 0$ and $D > 0$, the method $\mathtt{fom}(\epsilon)$ has
\begin{equation}
\label{eq:assum-decrement-for-fom}
\begin{cases}
f(x^{(0)}) - f_\star \ge \delta \\
\operatorname{dist}(x^{(0)}, X_\star) \le D
\end{cases}
\implies \min\{ f(x^{(k)}): 0\le k\le K_{\mathtt{dec}}(\epsilon, \delta, D) \} \le f(x^{(0)}) -\epsilon.
\end{equation}
\end{assumption}

Note that if~\eqref{eq:assum-decrement-for-fom} holds for a fixed tuple $(\hat{\epsilon}, \hat{\delta}, \hat{D})$, then for any $\tilde{\delta}\ge\hat{\delta}, \tilde{D}\le\hat{D}$, setting $K_{\mathtt{dec}}(\hat{\epsilon}, \tilde{\delta}, \tilde{D}) = K_{\mathtt{dec}}(\hat{\epsilon}, \hat{\delta}, \hat{D})$ also trivially satisfies~\eqref{eq:assum-decrement-for-fom} for $(\hat{\epsilon}, \tilde{\delta}, \tilde{D})$. Thus, without loss of generality, assume the function $K_{\mathtt{dec}}(\epsilon, \delta, D)$ is nonincreasing in $\delta$ and nondecreasing in $D$. We do not have such natural monotonicity in $\epsilon$, since $\epsilon$ is also a parameter of $\mathtt{fom}(\epsilon)$. Changing $\epsilon$ results in a different first-order method, so the corresponding $K_{\mathtt{dec}}$ are not comparable. Section~\ref{Sect:cor-monotone} provides some specific examples of $\mathtt{fom}$ with corresponding $K_{\mathtt{dec}}(\epsilon, \delta, D)$.\\

\noindent {\bf Function growth.} We consider objective functions with general growth structures in~\eqref{eq:intro-growth-structure}, formalized in the following assumption.
\begin{assumption}[Function growth]
\label{assumption:function-growth}
There exists a strictly increasing and invertible function $G:[0,+\infty) \to [0,+\infty)$ such that the following holds for $x\in \{z: f(z)\le f(x^{(0)})\}$
\begin{equation*}
f(x)-f_\star \ge G\left( \operatorname{dist}(x, X_\star) \right).
\end{equation*}
\end{assumption}
A typical choice of $G$ is $G(t) = \mu t^p$. In Section~\ref{Sect:cor-monotone}, we consider examples of this polynomial form.

\subsection{A Dynamic Restarting Scheme}

The closely related work~\cite{Renegar2022restart} introduced a restarting scheme that consists of multiple parallel processes, in which each process runs a copy of a first-order method $\mathtt{fom}$. In their scheme, the number of processes is determined a priori and remains fixed throughout execution. Moreover, each process is assigned an exponentially spaced decrement target. Once a process attains its target decrease, it restarts and its iterate is passed to other processes for their potential use.

Motivated by this framework, we propose a dynamic restarting scheme. Our scheme also consists of a number of parallel processes, each running a copy of $\mathtt{fom}(\epsilon)$. Specifically, we let the $k$-th process $P_k$ run $\mathtt{fom}(\epsilon_k)$ with a predetermined parameter $\epsilon_k>0$ from a strictly increasing sequence tending to infinity. Two key novelties of our scheme are that the number of active parallel processes adapts dynamically and the targets $\epsilon_k$ can be optimized over. These contrast~\cite{Renegar2022restart}, which fixes the number of processes and $\epsilon_k=2^{k-1}\epsilon$ in advance. In our scheme, we start with $N_0$ processes (by default, $N_0=1$) and add new processes as the scheme proceeds. We also simplify communication to use a centralized rule for passing iterates between processes, focusing only on the best newly generated point $\overline{x}^{(t)}$ at each iteration. 

\SetAlgoLined
\LinesNumbered
\SetAlgoLongEnd

\begin{algorithm}
\caption{Dynamic Restarting Scheme with First-Order Processes}
\label{algo:general-framework}
\KwIn{An initialization $x^{(0)}$, decrement targets $\{\epsilon_k\}_{k=0}^{\infty}$, a known first-order method $\mathtt{fom}(\epsilon)$, initial process count $N_0$ ($=1$ by default).}
Round $0$, initialization: For each $k=0,\dots,N_0-1$, launch a process $P_k$ and set a task for it.\\

\SetKwBlock{round}{For round $t=1,2,\dots$:}{end of a round}
\SetKwIF{If}{ElseIf}{Else}{if}{then}{else if}{else}{}
\SetKwBlock{steptwo}{Step 2:}{}

\round{
\textbf{Step 1:} Run one iteration for all processes in the scheme, and select an iterate $\overline{x}^{(t)}$ attaining the minimum objective value among all the new iterates in round $t$.\\

\steptwo(For each process $P_k$ in the scheme, check whether $\overline{x}^{(t)}$ satisfies the task for $P_k$.){
\If{the task for $P_k$ is accomplished}{
Restart $P_k$ at $\overline{x}^{(t)}$ and update the task for $P_k$.\\
\If{$P_k$ is the highest indexed process entering round $t$}{
Launch a new process $P_{k+1}$ and set a task for it.
}
}
}

}
\end{algorithm}

\noindent {\bf Basic Setup of the Scheme.} We assume all the current processes in the scheme iterate simultaneously. That is, all the processes run one iteration in each round in Algorithm~\ref{algo:general-framework}. Then the number of rounds in Algorithm~\ref{algo:general-framework} measures iteration complexity, while the total computational cost additionally depends on the number of processes (typically at most $\mathcal{O}(\log\log(1/\epsilon))$ in applications, see the corollaries in Section~\ref{Sect:cor-monotone}). Each round consists of two steps. We first run one iteration for all processes and collect the iterate $\overline{x}^{(t)}$ achieving the minimum objective value among all the new iterates. Then we check whether $\overline{x}^{(t)}$ satisfies any of the restart conditions, which are characterized by tasks. In the remainder of this section, we explain the details of the task and restart.

\begin{remark}
From an implementation perspective, the process-level iterations within a single round need not be executed in parallel. For example, one may sequentially generate each process's next iterate in a round, slowing the runtime, typically by a factor of $\mathcal{O}(\log\log(1/\epsilon))$. Our theory bounds iteration complexity, the total number of rounds, and oracle complexity, the total number of $\mathtt{fom}$ steps taken. Thus, whether parallel computing is used will determine which measure is most relevant in estimating runtime.
\end{remark}
\begin{remark}
Our model assumes that all the processes are synchronized, with each process running exactly one iteration in each round. One possible generalization is an asynchronous scheme, as studied by~\cite{Renegar2022restart}, where the processes are allowed to make iterations at different rates. Such an extension would be natural for our framework as well, but we restrict attention to the synchronous setting.
\end{remark}

\noindent {\bf Task and Restart.} Each process $P_k$ in the scheme has a ``task'' that determines when it restarts. The task for $P_k$ is defined by a ``reference point'' $(x_{\mathtt{ref}})_k^{(t)}$ and a decrement target $\epsilon_k$. For each $t\ge 1$ and any $P_k$ that is in the scheme at the beginning of round $t$, we use $(x_{\mathtt{ref}})_k^{(t)}$ to denote the reference point of $P_k$ entering round $t$. The task for $P_k$ is to obtain a point $\overline{x}$ such that $f(\overline{x}) \le f\left( (x_{\mathtt{ref}})_k^{(t)} \right) - \epsilon_k$. So the algorithm stores $f\left( (x_{\mathtt{ref}})_k^{(t)} \right)$ for $P_k$, and the task remains unchanged until the reference point is updated to a different point in some round. At initialization in Algorithm~\ref{algo:general-framework}, for each $k=0,\dots,N_0-1$, when the process $P_k$ is launched at the initial point $x^{(0)}$, we also set the initial reference point as $(x_{\mathtt{ref}})_k^{(1)} \gets x^{(0)}$. We say the first $N_0$ processes $P_0,\dots,P_{N_0-1}$ start at round $0$.

For each $t\ge 1$, at Step 2 of round $t$ in Algorithm~\ref{algo:general-framework}, we check the completion status of the task for each process, using the iterate $\overline{x}^{(t)}$ obtained from Step 1. If the task for $P_k$ is not accomplished, then no update to $P_k$ is needed, and we simply set $(x_{\mathtt{ref}})_k^{(t+1)} \gets (x_{\mathtt{ref}})_k^{(t)}$. If the task for $P_k$ is accomplished, we restart $P_k$ at $\overline{x}^{(t)}$, which formally entails the following update for $P_k$: we clear the memory of past iterates in $P_k$, set $\overline{x}^{(t)}$ as the new initial point of $P_k$, and continue running $\mathtt{fom}(\epsilon_k)$ in the subsequent rounds. Following the restart, $P_k$ will behave as if $\mathtt{fom}(\epsilon_k)$ were run from the new initial point $\overline{x}^{(t)}$. We also update the task for $P_k$ by setting $\overline{x}^{(t)}$ as the new reference point, that is, $(x_{\mathtt{ref}})_k^{(t+1)} \gets \overline{x}^{(t)}$. Moreover, if $P_k$ is also the ``highest'' indexed process in the scheme at the beginning of round $t$, which means the scheme contains exactly $(k+1)$ processes $P_0,\dots,P_{k}$ entering round $t$, then we add a new process $P_{k+1}$ into the scheme. In this case, we say the process $P_{k+1}$ starts at round $t$. In the next round $t+1$, the new process $P_{k+1}$ will perform its first iteration of $\mathtt{fom}(\epsilon_{k+1})$ from the initial point $\overline{x}^{(t)}$. When $P_{k+1}$ is initially launched, we also assign it the task with reference point $(x_{\mathtt{ref}})_{k+1}^{(t+1)} \gets \overline{x}^{(t)}$ and the decrement target $\epsilon_{k+1}$.

%% file: theory.tex
\section{Main Results and Analysis}
\label{Sect:main-results}

Our main result, Theorem~\ref{thm:main}, provides an iteration complexity bound for Algorithm~\ref{algo:general-framework} to compute a $(2\epsilon_0)$-optimal solution under Assumptions~\ref{assumption:decrement-for-fom} and~\ref{assumption:function-growth}. In our analysis, at most $m$ (defined below in~\eqref{eq:m}) processes are relevant to this iteration complexity. We also define a related quantity $\hat{m}$ to upper bound the total number of processes in the scheme.

Suppose the initial objective gap is $\Delta_0 = f(x^{(0)}) - f_\star$. A central idea of our restarting scheme is to let the decrement targets scale upward, so that they eventually surpass the unknown initial gap. To this end, we assume the decrement target sequence $\{\epsilon_k\}$ is strictly increasing with $\epsilon_k \to \infty$. Let $m$ be the first index such that $\epsilon_k$ reaches half the initial gap, and $\hat{m}$ the first to strictly exceed the full initial gap.
\begin{equation}
\label{eq:m}
m:= \inf \left\{k: \epsilon_k \ge \frac{\Delta_0}{2} \right\}, \quad\text{and}\quad \hat{m}:= \inf \left\{k: \epsilon_k > \Delta_0 \right\}.
\end{equation}
The core idea of our analysis is to treat the following two phases toward obtaining a $(2\epsilon_0)$-optimal solution, corresponding to the two parts in the statement of Theorem~\ref{thm:main}: \textit{Phase 1}, a burn-in phase during which the scheme launches up to $m$ processes, and \textit{Phase 2}, a convergence phase where every process $P_k$ restarts repeatedly and its objective gap shrinks to $\mathcal{O}(\epsilon_k)$. In each phase, the number of rounds between consecutive restarts of each process $P_k$ can be bounded by terms of the following form, assuming $\delta_\text{max} \ge \delta_\text{min} \ge 2\epsilon_k$ and $G$ is given or is clear from the context.
\begin{equation}
\label{eq:Q}
Q_k\left( \delta_\text{min}, \delta_\text{max} \right) := \sup_{\delta \in [\delta_\text{min}, \delta_\text{max} ]} K_{\mathtt{dec}} \left( \epsilon_k, \delta, G^{-1}(\delta) \right) .
\end{equation}
Given any first-order method and function growth structure satisfying Assumptions~\ref{assumption:decrement-for-fom} and~\ref{assumption:function-growth} respectively, and any strictly increasing sequence $\{\epsilon_k\}$ diverging to infinity, we have the following iteration complexity for computing a $(2\epsilon_0)$-optimal solution.

\begin{theorem}
\label{thm:main}
Suppose Assumption~\ref{assumption:decrement-for-fom} holds for $\mathtt{fom}(\epsilon)$ and Assumption~\ref{assumption:function-growth} holds for $f$. Then Algorithm~\ref{algo:general-framework} computes a $(2\epsilon_0)$-optimal solution no later than round
\begin{equation*}
\begin{split}
& \underbrace{\max_{N_0\le s\le m} \sum_{k= N_0-1}^{s-1} Q_k\left( \epsilon_{s-1}+ \sum_{j=k}^{s-1} \epsilon_j, \Delta_0- \sum_{j= N_0-1}^{k-1} \epsilon_j \right) }_{\text{Phase 1}} \\
&+ \underbrace{\sum_{k=0}^{m-1}\ \max_{1\le r < \frac{2\epsilon_{k+1}}{\epsilon_{k}} }\ \sum_{i=0}^{r-1} Q_k\left( (r+1-i) \epsilon_{k}, \min \{ 2\epsilon_{k+1} + (1-i) \epsilon_{k}, \Delta_0 \} \right) }_{\text{Phase 2}} ,
\end{split}
\end{equation*}
and at most $\max\{\hat{m}+1, N_0\}$ processes are launched.
\end{theorem}

Recall that $Q_k$ is defined in~\eqref{eq:Q} through the term $K_{\mathtt{dec}}(\epsilon, \delta, G^{-1}(\delta))$. All the $Q_k$ functions in Theorem~\ref{thm:main} simplify significantly if $K_{\mathtt{dec}}(\epsilon, \delta, G^{-1}(\delta))$ satisfies the monotonicity condition defined below in~\eqref{eq:cor-monotone-eq1}. In this case, the upper bound in Theorem~\ref{thm:main} simplifies.

\begin{corollary}
\label{cor:monotone}
Suppose Assumption~\ref{assumption:decrement-for-fom} holds for $\mathtt{fom}(\epsilon)$ and Assumption~\ref{assumption:function-growth} holds for $f$. If
\begin{equation}
\label{eq:cor-monotone-eq1}
K_{\mathtt{dec}}(\epsilon, \delta, G^{-1}(\delta)) \text{ is nonincreasing in } \delta \text{ on the interval } [ 2\epsilon, +\infty), \forall \epsilon>0,
\end{equation}
then Algorithm~\ref{algo:general-framework} computes a $(2\epsilon_0)$-optimal solution no later than round
\begin{equation}
\label{eq:cor-monotone-eq2}
\sum_{k=0}^{m-1} \left( 1+ \frac{2\epsilon_{k+1}}{\epsilon_{k}} \right) K_{\mathtt{dec}} \left( \epsilon_k, 2 \epsilon_{k}, G^{-1}(2 \epsilon_{k}) \right) .
\end{equation}
\end{corollary}

Section~\ref{Sect:cor-monotone} presents three examples satisfying~\eqref{eq:cor-monotone-eq1} where Corollary~\ref{cor:monotone} applies.

\subsection{Analysis of the Restarting Scheme}

Now, we analyze the general restarting scheme in Algorithm~\ref{algo:general-framework}. We first prove Lemma~\ref{lemma:number-of-processes}, which is an upper bound on the total number of processes in the scheme. Then we analyze the iteration complexity through the two phases described earlier. Lemma~\ref{lemma:T0-upper-bound} gives an upper bound for the number of rounds in Phase 1. Lemma~\ref{lemma:Tk-T'k-upper-bound} provides a useful result to analyze Phase 2. 

We introduce two useful notions $S_k$ and $\Delta_k$. The initial gap of each process will be used repeatedly in our analysis. This can be characterized by the reference point $(x_{\mathtt{ref}})_k^{(t)}$ together with the ``starting time'' $S_k$, where $S_k$ is the round in which $P_k$ starts. Note $S_k$ is well-defined if and only if $P_k$ appears in the scheme. Let $\Delta_k$ be the objective gap when $P_k$ is first launched. We additionally define $\overline{x}^{(0)} := x^{(0)}$. Considering round $S_k$, we see that the initial point of $P_k$ is $\overline{x}^{(S_k)}$. Another key observation here is that the initial reference point of each process is set as its initial point. Thus, the four notions $S_k$, $\overline{x}^{(S_k)}$, $(x_{\mathtt{ref}})_k^{(S_k+1)}$, and $\Delta_k$ relate as follows.
\begin{equation}
\label{eq:Delta_k}
\begin{split}
S_k :=&\, \inf\{t\ge 0: P_k \text{ is in the scheme by the end of round $t$} \} ,\\
\overline{x}^{(S_k)} =& \text{ the initial point of $P_k$} = (x_{\mathtt{ref}})_k^{(S_k+1)},\\
\Delta_k :=& \text{ the initial objective gap of $P_k$} = f(\overline{x}^{(S_k)}) -f_\star = f\left((x_{\mathtt{ref}})_k^{(S_k+1)}\right) - f_\star .
\end{split}
\end{equation}

\subsubsection{Number of Processes}

We show the following upper bound for the number of processes in the scheme. As noted earlier, this bound involves the quantity $\hat{m}$.

\begin{lemma}
\label{lemma:number-of-processes}
The number of processes in the scheme is at most $\max \{\hat{m}+1 , N_0\}$.
\end{lemma}

\begin{proof}[Proof of Lemma~\ref{lemma:number-of-processes}]
Consider the following two cases.

If $\hat{m}+1 < N_0$, by the definition of $\hat{m}$ in~\eqref{eq:m} and the monotonicity of $\{\epsilon_k\}$, we have $f\left((x_{\mathtt{ref}})_{N_0-1}^{(1)}\right) = f(x^{(0)}) < f_\star + \epsilon_{\hat{m}} \le f_\star + \epsilon_{N_0-1}$. So the process $P_{N_0-1}$ never restarts, and the number of processes is always $N_0$.

If $\hat{m}+1\ge N_0$, we need to prove the number of processes is at most $\hat{m}+1$. For any $k\ge N_0-1$, if $P_{k+1}$ appears in the scheme, then we have $f\left((x_{\mathtt{ref}})_{k+1}^{(S_{k+1}+1)}\right) \le f\left((x_{\mathtt{ref}})_k^{(S_k+1)}\right) -\epsilon_k < f\left((x_{\mathtt{ref}})_k^{(S_k+1)}\right)$, since the process $P_{k+1}$ is launched when $P_{k}$ first restarts. This inequality inductively implies $f\left((x_{\mathtt{ref}})_{\hat{m}}^{(S_{\hat{m}}+1)}\right) \le f\left((x_{\mathtt{ref}})_{N_0-1}^{(S_{N_0-1}+1)}\right) = f\left((x_{\mathtt{ref}})_{N_0-1}^{(1)}\right) = f(x^{(0)}) < f_\star + \epsilon_{\hat{m}}$, where the last step is by~\eqref{eq:m}. So the process $P_{\hat{m}}$ (if it ever appears in the scheme) never restarts, which implies $P_{\hat{m}+1}$ never appears in the scheme.
\end{proof}

Lemma~\ref{lemma:number-of-processes} directly implies the bound on the number of processes in Theorem~\ref{thm:main}.

\subsubsection{Number of Rounds in Phase 1}

We formally define ``Phase 1'' as all the rounds before the occurrence of a certain event, which will be precisely specified in~\eqref{eq:T_0} through a special quantity $n$. Note the concept of phases is introduced solely for expository clarity, while the algorithm does not need to know what ``phase'' it is in. Using $\Delta_k$, we define $n$ as follows. The well-definedness of $n$ follows directly from Lemma~\ref{lemma:n-is-welldefined}, which is proved later.
\begin{equation}
\label{eq:n}
n:=\inf \left\{k: \text{$P_k$ appears in the scheme and } \Delta_k \le 2\epsilon_k \right\}.
\end{equation}
So~\eqref{eq:n} considers the first process $P_k$ with an initial gap of at most $2\epsilon_k$. If $n=0$, this means the initial objective gap satisfies $\Delta_0 \le 2\epsilon_0$. Then Algorithm~\ref{algo:general-framework} computes a $(2\epsilon_0)$-optimal solution by round $0$, which is a trivial case of Theorem~\ref{thm:main}. Hereafter, we only consider $n\ge 1$.

As another fundamental component of our analysis, Assumption~\ref{assumption:decrement-for-fom} will be used frequently. Recall that Assumption~\ref{assumption:decrement-for-fom} gives the following implication for any $\delta \ge 2\epsilon > 0$ and $D >0$,
\begin{equation}
\label{eq:Kdec-implication}
\begin{cases}
f(x^{(0)}) - f_\star \ge \delta \\
\operatorname{dist}(x^{(0)}, X_\star) \le D
\end{cases}
\implies \min\{ f(x^{(k)}): 0\le k\le K_{\mathtt{dec}}(\epsilon, \delta, D) \} \le f(x^{(0)}) -\epsilon.
\end{equation}

\begin{lemma}
\label{lemma:n-is-welldefined}
Suppose Assumption~\ref{assumption:decrement-for-fom} holds for $\mathtt{fom}(\epsilon)$. Then $n$ is well-defined in~\eqref{eq:n}, that is, there exists some $k\ge 0$ such that $P_k$ appears in the scheme and $\Delta_k \le 2\epsilon_k$.
\end{lemma}

\begin{proof}[Proof of Lemma~\ref{lemma:n-is-welldefined}]
Suppose $f\left((x_{\mathtt{ref}})_k^{(S_k+1)}\right) - f_\star = \Delta_k> 2\epsilon_k$ for any $P_k$ that appears in the scheme. Let $D_k= \operatorname{dist}\left((x_{\mathtt{ref}})_k^{(S_k+1)}, X_\star\right)$. By~\eqref{eq:Kdec-implication}, $P_k$ itself will produce an iterate satisfying its task after at most $K_{\mathtt{dec}}(\epsilon_k, \Delta_k, D_k)$ rounds even without intervention from other processes. So $P_k$ will eventually restart, leading to the launch of $P_{k+1}$. This inductively implies that $P_k$ appears in the scheme for any $k\ge 0$, which contradicts Lemma~\ref{lemma:number-of-processes}.
\end{proof}

Since we have just established the existence of $P_n$ in Lemma~\ref{lemma:n-is-welldefined}, let $T_0<\infty$ be the round when $P_n$ is launched, which can be characterized as
\begin{equation}
\label{eq:T_0}
T_0 := \inf\{t\ge 0: P_n \text{ is in the scheme by the end of round $t$} \} = S_n.
\end{equation}
As mentioned earlier, we formally define ``Phase 1'' as the whole phase before $P_n$ is launched, which contains $T_0$ rounds. We have the following result for $T_0$. Note, by the convention stated in Section~\ref{Sect:Prelim}, the upper bound below becomes $0$ if $n\le N_0-1$.

\begin{lemma}
\label{lemma:T0-upper-bound}
Suppose Assumption~\ref{assumption:decrement-for-fom} holds for $\mathtt{fom}(\epsilon)$ and Assumption~\ref{assumption:function-growth} holds for $f$. Then the process $P_n$ is launched no later than round
\begin{equation*}
T_0 \le \sum_{k= N_0-1}^{n-1} K_{\mathtt{dec}} \left( \epsilon_k, \Delta_k, G^{-1}(\Delta_k) \right).
\end{equation*}
\end{lemma}

\begin{proof}[Proof of Lemma~\ref{lemma:T0-upper-bound}]
Consider the following two cases.

(i) If $n\le N_0-1$, then $P_n$ is launched at the initialization, so $T_0=0$.

(ii) If $n\ge N_0$, then by the definition of $n$ in~\eqref{eq:n}, we have $f\left((x_{\mathtt{ref}})_k^{(S_k+1)}\right) - f_\star = \Delta_k > 2\epsilon_k$ for any $N_0-1\le k\le n-1$. Then by~\eqref{eq:Kdec-implication}, for $N_0-1\le k\le n-1$, the time between the launches of $P_k$ and $P_{k+1}$ is at most $K_{\mathtt{dec}}\left(\epsilon_k, \Delta_k, \operatorname{dist}((x_{\mathtt{ref}})_k^{(S_k+1)}, X_\star)\right)$, which is further upper bounded by $K_{\mathtt{dec}} \left( \epsilon_k, \Delta_k, G^{-1}(\Delta_k) \right)$. So the total rounds before the launch of $P_n$ is at most $\sum_{k= N_0-1}^{n-1} K_{\mathtt{dec}} \left( \epsilon_k, \Delta_k, G^{-1}(\Delta_k) \right)$.
\end{proof}

Lemma~\ref{lemma:T0-upper-bound} establishes an upper bound on the number of rounds in Phase 1, namely the number of rounds needed to launch the process $P_n$.

\subsubsection{Number of Rounds in Phase 2}

In a similar fashion as $T_0$, we analyze the number of rounds needed for certain events to occur. ``Phase 2'' will also be formally defined via one of these special events.

We use ``(re)start'' to mean either a restart or a start/launch. For $k=1,\dots,n$, we recursively define the round $T'_k$ to be the latest round no later than $T_{k-1}$ when $P_{n-k}$ (re)starts, and $T_k$ as the earliest round not preceding $T'_{k}$ when $P_{n-k}$ (re)starts with objective gap at most $2\epsilon_{n-k}$, which can be formally written as
\begin{equation}
\label{eq:T'_k}
T'_k := \sup\{0\le t\le T_{k-1}: \text{$P_{n-k}$ (re)starts at round $t$} \}, \quad \forall 1\le k\le n.
\end{equation}
\begin{equation}
\label{eq:T_k}
T_k := \inf\{t\ge T'_{k}: \text{$P_{n-k}$ (re)starts at round $t$, and } f(\overline{x}^{(t)})-f_\star\le 2\epsilon_{n-k} \}, \quad \forall 1\le k\le n.
\end{equation}
The existence of $T'_k$ and $T_k$ can be established inductively. Suppose $T_{k-1}<\infty$ is well-defined, then $P_{n-k+1}$ is in the scheme at round $T_{k-1}$, and so is $P_{n-k}$. So the set in~\eqref{eq:T'_k} is nonempty and bounded, which ensures $T'_k<\infty$ exists. By~\eqref{eq:Kdec-implication}, $P_{n-k}$ will eventually restart if its objective gap is larger than $2\epsilon_{n-k}$, with the objective gap reducing by at least $\epsilon_{n-k}$ at each restart. Thus, from round $T'_k$ onward, $P_{n-k}$ will eventually restart with an objective gap at most $2\epsilon_{n-k}$, which guarantees $T_k<\infty$ is well-defined by~\eqref{eq:T_k}.

Algorithm~\ref{algo:general-framework} is guaranteed to compute a $(2\epsilon_0)$-optimal solution at round $T_n$, since $T_n$ is defined to be a round at which $P_0$ (re)starts with objective gap at most $2\epsilon_{0}$. We already have an upper bound for $T_0$ from the analysis for Phase 1, so it remains to control $T_n-T_0$. Now we formally define ``Phase 2'' as the whole phase between round $T_0$ and round $T_n$. Phase 2 contains $T_n-T_0$ rounds if $T_0\le T_n$. Even if $T_0>T_n$, we still focus on upper bounding $T_n-T_0$, and all the analysis below remains valid. Observe that
\begin{equation}
\label{eq:phase2-decompose}
T_n - T_0 = \sum_{k=1}^{n} (T_{k}- T_{k-1})\le \sum_{k=1}^{n} (T_{k}- T'_{k}) ,
\end{equation}
then it suffices to upper bound the terms $T_{k}- T'_{k}$. Such bounds can be established through $Q_k$.

\begin{lemma}
\label{lemma:Tk-T'k-upper-bound}
Suppose Assumption~\ref{assumption:decrement-for-fom} holds for $\mathtt{fom}(\epsilon)$ and Assumption~\ref{assumption:function-growth} holds for $f$. Then for any $1\le k\le n$,
\begin{equation}
\label{eq:lemma-Tk-T'k-upper-bound-eq1}
T_k - T'_{k} \le \max_{1\le r < \frac{2\epsilon_{n-k+1}}{\epsilon_{n-k}} }\ \sum_{i=0}^{r-1} Q_{n-k}\left( (r+1-i) \epsilon_{n-k}, \min \{ 2\epsilon_{n-k+1} + (1-i) \epsilon_{n-k}, \Delta_0 \} \right).
\end{equation}
\end{lemma}

\begin{proof}[Proof of Lemma~\ref{lemma:Tk-T'k-upper-bound}]
Consider an arbitrary fixed $1\le k\le n$. By the definition of $T_{k-1}$ in~\eqref{eq:T_k} (or by the definition in~\eqref{eq:n} and~\eqref{eq:T_0} if $k=1$), the process $P_{n-k+1}$ (re)starts in round $T_{k-1}$ at point $\overline{x}^{(T_{k-1})}$ with objective gap $f( \overline{x}^{(T_{k-1})} ) -f_\star \le 2\epsilon_{n-k+1}$. Recall the two steps in Algorithm~\ref{algo:general-framework}, $\overline{x}^{(T_{k-1})}$ is the point obtained from Step 1 of round $T_{k-1}$. Consider the following two cases in round $T_{k-1}$.
\begin{itemize}
\item If $P_{n-k}$ also (re)starts at $\overline{x}^{(T_{k-1})}$ in round $T_{k-1}$, then by the definition of $T'_k$ in~\eqref{eq:T'_k}, we have $T'_k = T_{k-1}$ and $\overline{x}^{(T'_{k})} = \overline{x}^{(T_{k-1})}$. So $P_{n-k}$ (re)starts in round $T'_k$ with objective gap $f( \overline{x}^{(T'_{k})} ) -f_\star \le 2\epsilon_{n-k+1}$.
\item If $P_{n-k}$ does not (re)start in round $T_{k-1}$, then the latest reference point of $P_{n-k}$ is $(x_{\mathtt{ref}})_{n-k}^{(T_{k-1})} = \overline{x}^{(T'_{k})}$, since $T'_k$ is the latest round prior to $T_{k-1}$ when $P_{n-k}$ (re)starts. By the restart condition, the task for $P_{n-k}$ is not accomplished in round $T_{k-1}$, so $f(\overline{x}^{(T'_{k})}) - f( \overline{x}^{(T_{k-1})} ) \le \epsilon_{n-k}$. This further implies that $P_{n-k}$ (re)starts in round $T'_k$ with objective gap $f(\overline{x}^{(T'_{k})}) - f_\star\le 2\epsilon_{n-k+1}+ \epsilon_{n-k}$.
\end{itemize}
In either case, $P_{n-k}$ (re)starts in round $T'_{k}$ with objective gap at most $2\epsilon_{n-k+1}+ \epsilon_{n-k}$. Since $P_{n-k}$ also (re)starts at round $T_k$ and $T_k \ge T'_{k}$, we consider all the (re)starts of $P_{n-k}$ between $T'_{k}$ and $T_{k}$. Suppose $P_{n-k}$ (re)starts $r+1$ times in this interval (for some $r\ge 0$), and they happen at rounds $T'_{k} = T^{(0)} < \dots < T^{(r)} = T_{k}$. 
The case $r=0$ trivially has $T'_{k} = T_k$ and satisfies the desired result, because the right-hand side of~\eqref{eq:lemma-Tk-T'k-upper-bound-eq1} is always well-defined and non-negative by~\eqref{eq:proof-lemma-Tk-T'k-upper-bound-eq5}, which will be shown later in this proof. We only consider $r\ge 1$ in the rest of this proof. For $0\le i\le r$, suppose the objective gap of $P_{n-k}$ right after the (re)start at round $T^{(i)}$ is $\delta_i$, namely $\delta_i = f(\overline{x}^{(T^{(i)})}) - f_\star$. 
By the restart condition, every time $P_{n-k}$ restarts, its objective gap decreases by at least $\epsilon_{n-k}$, which implies 
\begin{equation}
\label{eq:proof-lemma-Tk-T'k-upper-bound-eq2}
\delta_i - \delta_{i+1} \ge \epsilon_{n-k}, \quad \forall 0\le i\le r-1.
\end{equation}
We also have $\delta_0 \le 2\epsilon_{n-k+1}+ \epsilon_{n-k}$ by previous analysis of $T'_{k}$, and $\delta_r \le 2\epsilon_{n-k} < \delta_{r-1}$ by the definition of $T_k$ in~\eqref{eq:T_k}. Thus,
\begin{equation}
\label{eq:proof-lemma-Tk-T'k-upper-bound-eq3}
2\epsilon_{n-k+1}+ \epsilon_{n-k} \ge \delta_0 > \dots > \delta_{r-1} > 2\epsilon_{n-k} \ge \delta_r .
\end{equation}
So $\delta_i > 2\epsilon_{n-k}$ for any $0\le i\le r-1$. Then by~\eqref{eq:Kdec-implication}, even without intervention from other processes, $P_{n-k}$ itself will produce an iterate satisfying its own task within at most $K_{\mathtt{dec}}\left(\epsilon_{n-k}, \delta_i, \operatorname{dist}( \overline{x}^{(T^{(i)})}, X_\star) \right)$ additional rounds after round $T^{(i)}$. So the following holds for $0\le i\le r-1$.
\begin{equation}
\label{eq:proof-lemma-Tk-T'k-upper-bound-eq1}
T^{(i+1)} - T^{(i)} \le K_{\mathtt{dec}} \left( \epsilon_{n-k}, \delta_i, \operatorname{dist}( \overline{x}^{(T^{(i)})}, X_\star) \right) \le K_{\mathtt{dec}} \left( \epsilon_{n-k}, \delta_i, G^{-1}( \delta_i) \right).
\end{equation}
For each $0\le i\le r-1$, note that $\delta_i = \delta_0 + \sum_{j=0}^{i-1} (\delta_{j+1}-\delta_{j}) = \delta_{r-1} + \sum_{j=i}^{r-2} (\delta_{j}-\delta_{j+1})$, then by~\eqref{eq:proof-lemma-Tk-T'k-upper-bound-eq2} and~\eqref{eq:proof-lemma-Tk-T'k-upper-bound-eq3}, we have that
\begin{equation}
\label{eq:proof-lemma-Tk-T'k-upper-bound-eq4}
(r+1-i) \epsilon_{n-k} < \delta_i \le 2\epsilon_{n-k+1} + (1-i) \epsilon_{n-k}, \quad \forall 0\le i\le r-1 .
\end{equation}
Since each $\delta_i$ is the objective gap of some process when it (re)starts, we immediately have $\delta_i \le \Delta_0$ since the objective gap strictly decreases at each restart. So 
\begin{equation}
\label{eq:proof-lemma-Tk-T'k-upper-bound-eq5}
(r+1-i) \epsilon_{n-k} < \delta_i \le \min \{ 2\epsilon_{n-k+1} + (1-i) \epsilon_{n-k}, \Delta_0 \}, \quad \forall 0\le i\le r-1 .
\end{equation}
Recall the notion of $Q_k$ from~\eqref{eq:Q}, then~\eqref{eq:proof-lemma-Tk-T'k-upper-bound-eq1} implies
\begin{equation*}
T^{(i+1)} - T^{(i)} \le Q_{n-k}\left( (r+1-i) \epsilon_{n-k}, \min \{ 2\epsilon_{n-k+1} + (1-i) \epsilon_{n-k}, \Delta_0 \} \right), \quad \forall 0\le i\le r-1.
\end{equation*}
From~\eqref{eq:proof-lemma-Tk-T'k-upper-bound-eq4} we also have $r \epsilon_{n-k} < 2\epsilon_{n-k+1}$, so $r < \frac{2\epsilon_{n-k+1}}{\epsilon_{n-k}}$. Combining this with the previous inequality gives
\begin{align*}
T_k - T'_{k} &= \sum_{i=0}^{r-1} (T^{(i+1)} - T^{(i)}) \\
&\le \max_{1\le r < \frac{2\epsilon_{n-k+1}}{\epsilon_{n-k}} }\ \sum_{i=0}^{r-1} Q_{n-k}\left( (r+1-i) \epsilon_{n-k}, \min \{ 2\epsilon_{n-k+1} + (1-i) \epsilon_{n-k}, \Delta_0 \} \right). \qedhere
\end{align*}
\end{proof}

Applying Lemma~\ref{lemma:Tk-T'k-upper-bound} to~\eqref{eq:phase2-decompose} upper bounds $T_n-T_0$, bounding the length of Phase 2. 

\subsection{Proof of Main Results}
In this part, we provide the proofs for Theorem~\ref{thm:main} and Corollary~\ref{cor:monotone}.

\subsubsection{Proof of Theorem~\ref{thm:main}}

We now use Lemma~\ref{lemma:T0-upper-bound} and Lemma~\ref{lemma:Tk-T'k-upper-bound} to control the number of rounds in Phase 1 and Phase 2, respectively, leading to the upper bound stated in Theorem~\ref{thm:main}.

\begin{proof}[Proof of Theorem~\ref{thm:main}]
The bound on the number of processes is given by Lemma~\ref{lemma:number-of-processes}. By the definition of $T_n$ in~\eqref{eq:T_k}, the process $P_0$ (re)starts with objective gap at most $2\epsilon_0$ at round $T_n$, which produces a $(2\epsilon_0)$-optimal solution. So we only need to prove the desired upper bound for $T_n$.

By the definition of $T'_k$ in~\eqref{eq:T'_k}, we have $T'_k\le T_{k-1}$ for any $1\le k\le n$. So
\begin{equation}
\label{eq:proof-thm-main-eq1}
T_n = T_0 + \sum_{k=1}^{n} (T_{k}- T_{k-1}) \le T_0 + \sum_{k=1}^{n} (T_{k}- T'_{k}) .
\end{equation}
By Lemma~\ref{lemma:Tk-T'k-upper-bound}, the following holds for any $1\le k\le n$. 
\begin{equation*}
T_k - T'_{k} \le \max_{1\le r < \frac{2\epsilon_{n-k+1}}{\epsilon_{n-k}} }\ \sum_{i=0}^{r-1} Q_{n-k}\left( (r+1-i) \epsilon_{n-k}, \min \{ 2\epsilon_{n-k+1} + (1-i) \epsilon_{n-k}, \Delta_0 \} \right) .
\end{equation*}
Summing over $k=1,\dots,n$ and reindexing, we obtain
\begin{equation*}
\begin{split}
\sum_{k=1}^{n} (T_{k}- T'_{k}) &\le \sum_{k=1}^{n}\ \max_{1\le r < \frac{2\epsilon_{n-k+1}}{\epsilon_{n-k}} }\ \sum_{i=0}^{r-1} Q_{n-k}\left( (r+1-i) \epsilon_{n-k}, \min \{ 2\epsilon_{n-k+1} + (1-i) \epsilon_{n-k}, \Delta_0 \} \right) \\
&= \sum_{k=0}^{n-1}\ \max_{1\le r < \frac{2\epsilon_{k+1}}{\epsilon_{k}} }\ \sum_{i=0}^{r-1} Q_k\left( (r+1-i) \epsilon_{k}, \min \{ 2\epsilon_{k+1} + (1-i) \epsilon_{k}, \Delta_0 \} \right) .
\end{split}
\end{equation*}
By the definition of $n$ in~\eqref{eq:n}, we have $2\epsilon_{n-1} < \Delta_{n-1} \le \Delta_0$, since $\Delta_k$ is nonincreasing in $k$. Comparing this with the definition of $m$ in~\eqref{eq:m} gives $n-1\le m-1$. So the previous inequality can be further upper bounded as follows.
\begin{equation}
\label{eq:proof-thm-main-eq2}
\sum_{k=1}^{n} (T_{k}- T'_{k}) \le \sum_{k=0}^{m-1}\ \max_{1\le r < \frac{2\epsilon_{k+1}}{\epsilon_{k}} }\ \sum_{i=0}^{r-1} Q_k\left( (r+1-i) \epsilon_{k}, \min \{ 2\epsilon_{k+1} + (1-i) \epsilon_{k}, \Delta_0 \} \right) .
\end{equation}
It remains to upper bound the term $T_0$ in~\eqref{eq:proof-thm-main-eq1}, which suffices to prove
\begin{equation}
\label{eq:proof-thm-main-eq3}
T_0 \le \max_{N_0\le s\le m} \sum_{k= N_0-1}^{s-1} Q_k\left( \epsilon_{s-1}+ \sum_{j=k}^{s-1} \epsilon_j, \Delta_0- \sum_{j= N_0-1}^{k-1} \epsilon_j \right) .
\end{equation}
We consider the following two cases.

(i) If $\Delta_0 \le 2\epsilon_{N_0-1}$, then $\Delta_{N_0-1} = f(\overline{x}^{(0)}) - f_\star = \Delta_0 \le 2\epsilon_{N_0-1}$ since $P_{N_0-1}$ is initialized at point $\overline{x}^{(0)}$ in Algorithm~\ref{algo:general-framework}. Then we have $n\le N_0-1$ by the definition of $n$ in~\eqref{eq:n}. The upper bound in Lemma~\ref{lemma:T0-upper-bound} is a summation over an empty set, which implies $T_0 =0$. By the definition of $m$ in~\eqref{eq:m}, we also have $m\le N_0-1$ since $\Delta_0 \le 2\epsilon_{N_0-1}$. So the right-hand side of~\eqref{eq:proof-thm-main-eq3} is a maximum over an empty set, which is also $0$. This proves~\eqref{eq:proof-thm-main-eq3} since both sides are $0$.

(ii) If $\Delta_0 > 2\epsilon_{N_0-1}$, then $\Delta_k = \Delta_0 > 2\epsilon_{N_0-1} \ge 2\epsilon_k$ for any $0\le k\le N_0-1$. So $n\ge N_0$ by the definition in~\eqref{eq:n}. For any $N_0-1 \le k \le n-1$, $P_{k+1}$ appears in the scheme (since $P_n$ appears, see~\eqref{eq:n}), but is not launched at the initialization of Algorithm~\ref{algo:general-framework}. Then by the restart condition and the definition of $\Delta_k$ in~\eqref{eq:Delta_k}, every time such a new process $P_{k+1}$ is launched, the new objective gap $\Delta_{k+1}$ is smaller than $\Delta_k$ by at least $\epsilon_k$:
\begin{equation*}
\Delta_{k} - \Delta_{k+1} \ge \epsilon_{k}, \quad \forall N_0-1 \le k\le n-1 .
\end{equation*}
From~\eqref{eq:n} we have $\Delta_{n-1} > 2\epsilon_{n-1}$. For each $N_0-1 \le k\le n-1$, noting $\Delta_k = \Delta_{N_0-1} + \sum_{j= N_0-1}^{k-1} (\Delta_{j+1}-\Delta_{j}) = \Delta_{n-1} + \sum_{j=k}^{n-2} (\Delta_{j}-\Delta_{j+1})$, we obtain 
\begin{equation*}
\epsilon_{n-1}+ \sum_{j=k}^{n-1} \epsilon_j = 2\epsilon_{n-1}+ \sum_{j=k}^{n-2} \epsilon_j < \Delta_k \le \Delta_{N_0-1}- \sum_{j= N_0-1}^{k-1} \epsilon_j = \Delta_0- \sum_{j= N_0-1}^{k-1} \epsilon_j, \quad \forall N_0-1\le k\le n-1.
\end{equation*}
Recall the notion of $Q_k$ from~\eqref{eq:Q}, then by Lemma~\ref{lemma:T0-upper-bound},
\begin{equation}
\label{eq:proof-thm-main-eq4}
T_0 \le \sum_{k= N_0-1}^{n-1} K_{\mathtt{dec}} \left( \epsilon_k, \Delta_k, G^{-1}(\Delta_k) \right) \le \sum_{k= N_0-1}^{n-1} Q_k\left( \epsilon_{n-1}+ \sum_{j=k}^{n-1} \epsilon_j, \Delta_0- \sum_{j= N_0-1}^{k-1} \epsilon_j \right).
\end{equation}
As noted earlier, $N_0\le n\le m$, so~\eqref{eq:proof-thm-main-eq4} implies~\eqref{eq:proof-thm-main-eq3}.

In both cases (i) and (ii), inequality~\eqref{eq:proof-thm-main-eq3} holds. Using~\eqref{eq:proof-thm-main-eq2} and~\eqref{eq:proof-thm-main-eq3} to upper bound the two parts in~\eqref{eq:proof-thm-main-eq1} directly gives the desired result.
\end{proof}

\subsubsection{Proof of Corollary~\ref{cor:monotone}}

By leveraging the monotonicity condition~\eqref{eq:cor-monotone-eq1}, all the $Q_k$ functions in Theorem~\ref{thm:main} simplify and combine into the single summation form in Corollary~\ref{cor:monotone}.

\begin{proof}[Proof of Corollary~\ref{cor:monotone}]
For any $\delta\ge 2\epsilon_k >0$, denote $Q_k^*(\delta) := K_{\mathtt{dec}} \left( \epsilon_k, \delta, G^{-1}(\delta) \right)$. By~\eqref{eq:cor-monotone-eq1}, $Q_k^*(\delta)$ is nonincreasing in $\delta$, so the following holds for any $\delta_\text{max} \ge \delta_\text{min} \ge 2\epsilon_k >0$,
\begin{equation*}
Q_k\left( \delta_\text{min}, \delta_\text{max} \right) = Q_k^*(\delta_\text{min}) .
\end{equation*}
Therefore, Theorem~\ref{thm:main} gives
\begin{equation}
\label{eq:proof-cor-monotone-eq1}
\begin{split}
T_n \le& \max_{N_0\le s\le m} \sum_{k= N_0-1}^{s-1} Q_k\left( \epsilon_{s-1}+ \sum_{j=k}^{s-1} \epsilon_j, \Delta_0- \sum_{j= N_0-1}^{k-1} \epsilon_j \right) \\
&+ \sum_{k=0}^{m-1}\ \max_{1\le r < \frac{2\epsilon_{k+1}}{\epsilon_{k}} }\ \sum_{i=0}^{r-1} Q_k\left( (r+1-i) \epsilon_{k}, \min \{ 2\epsilon_{k+1} + (1-i) \epsilon_{k}, \Delta_0 \} \right) \\
&= \underbrace{\max_{N_0\le s\le m} \sum_{k= N_0-1}^{s-1} Q_k^* \left( \epsilon_{s-1}+ \sum_{j=k}^{s-1} \epsilon_j \right) }_{(A)} + \underbrace{\sum_{k=0}^{m-1}\ \max_{1\le r < \frac{2\epsilon_{k+1}}{\epsilon_{k}} }\ \sum_{i=0}^{r-1} Q_k^* \left( (r+1-i) \epsilon_{k} \right) }_{(B)} .
\end{split}
\end{equation}
To bound the term $(A)$, we consider two cases for $m$. If $m\ge N_0$, then we have $\epsilon_{s-1} + \sum_{j=k}^{s-1} \epsilon_j \ge 2\epsilon_{s-1} \ge 2\epsilon_k$ for any $k\le s-1$. The monotonicity and non-negativity of $Q_k^* (\delta)$ give
\begin{equation}
\label{eq:proof-cor-monotone-eq2}
\max_{N_0\le s\le m} \sum_{k=N_0-1}^{s-1} Q_k^* \left( \epsilon_{s-1}+ \sum_{j=k}^{s-1} \epsilon_j \right) \le \max_{N_0\le s\le m} \sum_{k=N_0-1}^{s-1} Q_k^* \left( 2\epsilon_{k} \right) \le \sum_{k=0}^{m-1} Q_k^* \left( 2\epsilon_{k} \right) .
\end{equation}
If $m\le N_0-1$, then the term $(A)$ is a maximum over an empty set, so $(A)=0$ and~\eqref{eq:proof-cor-monotone-eq2} still holds. In either case, we have~\eqref{eq:proof-cor-monotone-eq2} as an upper bound for $(A)$. 

For the term $(B)$, since $\sum_{i=0}^{r-1} Q_k^* \left( (r+1-i) \epsilon_{k} \right) = \sum_{i=2}^{r+1} Q_k^* \left(i \epsilon_{k} \right)$ is nondecreasing in $r$, we have
\begin{equation}
\label{eq:proof-cor-monotone-eq3}
\begin{split}
\sum_{k=0}^{m-1}\ \max_{1\le r < \frac{2\epsilon_{k+1}}{\epsilon_{k}} }\ \sum_{i=0}^{r-1} Q_k^* \left( (r+1-i) \epsilon_{k} \right) \le \sum_{k=0}^{m-1} \sum_{i=2}^{\left\lfloor \frac{2\epsilon_{k+1}}{\epsilon_{k}} \right\rfloor +1} Q_k^* \left( i \epsilon_{k} \right) \le \sum_{k=0}^{m-1} \frac{2\epsilon_{k+1}}{\epsilon_{k}} Q_k^* \left( 2 \epsilon_{k} \right) ,
\end{split}
\end{equation}
where the last step is by the monotonicity of $Q_k^*(\delta)$. Using~\eqref{eq:proof-cor-monotone-eq2} and~\eqref{eq:proof-cor-monotone-eq3} to upper bound the two parts in~\eqref{eq:proof-cor-monotone-eq1}, we obtain
\begin{equation*}
T_n\le \sum_{k=0}^{m-1} \left( 1+ \frac{2\epsilon_{k+1}}{\epsilon_{k}} \right) Q_k^* \left( 2 \epsilon_{k} \right) = \sum_{k=0}^{m-1} \left( 1+ \frac{2\epsilon_{k+1}}{\epsilon_{k}} \right) K_{\mathtt{dec}} \left( \epsilon_k, 2 \epsilon_{k}, G^{-1}(2 \epsilon_{k}) \right) . \qedhere
\end{equation*}
\end{proof} 

\section{Parameter-Free Applications under $p\ge 2$ Growth}
\label{Sect:cor-monotone}

In this section, we apply Corollary~\ref{cor:monotone} to three general problem classes studied in~\cite{Renegar2022restart}: smooth functions, smoothable nonsmooth functions, and Lipschitz functions. For each class, we consider a specific example of a first-order method $\mathtt{fom}$ (with corresponding $K_{\mathtt{dec}}$), and a wide range of polynomial growth functions $G(t)=\mu\cdot t^p$ with $p\ge 2$. In prior work~\cite[Corollary 5.2]{Renegar2022restart}, the iteration complexity obtained for each example contains an additive burn-in term that is nonasymptotically significant, and the number of processes is of order $\mathcal{O}(\log(1/\epsilon))$. Our results improve upon both limitations, achieving the same asymptotic iteration complexity as~\cite{Renegar2022restart} without the burn-in term, while requiring only $\mathcal{O}(\log\log(1/\epsilon))$ processes.

As applications of our unified theory, we show that the monotonicity condition~\eqref{eq:cor-monotone-eq1} is satisfied in each of these three examples. Direct application of Corollary~\ref{cor:monotone} gives iteration complexity for a generic sequence $\{\epsilon_k\}$. Combining this with the bound on the number of processes from Lemma~\ref{lemma:number-of-processes}, we obtain the total first-order oracle complexity. One may select the values of $\{\epsilon_k\}$ to obtain the optimized iteration and oracle complexities derivable from these upper bounds. The following propositions will be invoked repeatedly to obtain such an optimized sequence $\{\epsilon_k\}$. Note that the following results optimize over sequences $\{\epsilon_k\}$, with quantities $m$ and $\hat{m}$ depending on $\{\epsilon_k\}$ as defined in~\eqref{eq:m}.

\begin{proposition}
\label{prop:linear-rate}
For any fixed $a>0$ and $0<\epsilon< \Delta_0$ with a sufficiently large ratio $\frac{\Delta_0}{\epsilon}$, 
\begin{equation*}
\inf_{\epsilon_0\le \frac{\epsilon}{2},\, \epsilon_k \uparrow \infty} \sum_{k=0}^{m-1} \frac{\epsilon_{k+1}^a}{\epsilon_{k}^a} = \Theta\left( \log \left(\frac{\Delta_0}{\epsilon} \right) \right), \quad\text{and}\quad \inf_{\epsilon_0\le \frac{\epsilon}{2},\, \epsilon_k \uparrow \infty} (\hat{m}+1) \cdot \sum_{k=0}^{m-1} \frac{\epsilon_{k+1}^a}{\epsilon_{k}^a} = \Theta\left( \left( \log \left(\frac{\Delta_0}{\epsilon} \right) \right)^2 \right) .
\end{equation*}
In particular, setting $\epsilon_k = \frac{\epsilon}{2}\cdot c^k$ for any $c>1$ attains these two asymptotic rates, with $\hat{m}= \mathcal{O}\left( \log \left(\frac{\Delta_0}{\epsilon} \right) \right)$ for this choice.
\end{proposition}

\begin{proposition}
\label{prop:sublinear-rate}
For any fixed $0<a<b$, and $0<\epsilon< \Delta_0$ with a sufficiently large ratio $\frac{\Delta_0}{\epsilon}$,
\begin{equation*}
\inf_{\epsilon_0\le \frac{\epsilon}{2},\, \epsilon_k \uparrow \infty} \sum_{k=0}^{m-1} \frac{\epsilon_{k+1}^a}{\epsilon_{k}^b} = \Theta\left( \frac{1}{\epsilon^{b-a}} \right) , \quad\text{and}\quad \inf_{\epsilon_0\le \frac{\epsilon}{2},\, \epsilon_k \uparrow \infty} (\hat{m}+1) \cdot \sum_{k=0}^{m-1} \frac{\epsilon_{k+1}^a}{\epsilon_{k}^b} = \Theta\left( \frac{1}{\epsilon^{b-a}} \cdot\log\log \left(\frac{\Delta_0}{\epsilon} \right) \right) .
\end{equation*}
In particular, setting $\epsilon_k = \frac{\epsilon}{2e}\cdot \exp(c^k)$ for any $c\in (1, \frac{b}{a})$ attains these two asymptotic rates, with $\hat{m}= \mathcal{O}\left( \log\log \left(\frac{\Delta_0}{\epsilon} \right) \right)$ for this choice.
\end{proposition}

The proofs of these two propositions are deferred to Appendix~\ref{Sect:technical-propositions}.

\subsection{Application to the Smooth Convex Setting}

For the class of $L$-smooth convex functions, we may consider accelerated methods, including the classic one from Nesterov~\cite{Nesterov1983} and the more recent and widely used FISTA~\cite{Beck2009fista}. In the rest of the paper, the ``accelerated method'' refers to the method defined below in~\eqref{eq:accel}, initialized at $x^{(0)}$.
\begin{equation}
\label{eq:accel}
\begin{split}
&\text{Initialize: } \theta^{(0)} = 1, y^{(0)}= x^{(0)}. \\
&\text{Iterate:} \begin{cases}
x^{(k+1)} = y^{(k)}- \frac{1}{L} \nabla f(y^{(k)}) ,\\
\theta^{(k+1)} = \frac{1+ \sqrt{ 1+4(\theta^{(k)})^2}}{2} ,\\
y^{(k+1)} = x^{(k+1)} + \frac{ \theta^{(k)}-1}{ \theta^{(k+1)}} (x^{(k+1)}- x^{(k)}) .
\end{cases}
\end{split}
\end{equation}
Note that the accelerated method does not depend on $\epsilon$. When applying this in Algorithm~\ref{algo:general-framework}, we simply let all $\mathtt{fom}(\epsilon_k)$ be the accelerated method, so all processes run the same method. The only differences among the processes are (i) their decrement targets $\epsilon_k$, and, as a consequence, (ii) their iterates in the same round may differ due to different restarting conditions.

For the accelerated method, we have the classic result $f(x^{(n)})-f_\star \le \frac{2L \operatorname{dist}(x^{(0)}, X_\star)^2}{(n+1)^2}$, which implies that Assumption~\ref{assumption:decrement-for-fom} holds with $K_{\mathtt{dec}}(\epsilon, \delta, D) = \sqrt{\frac{2LD^2}{\delta-\epsilon}}$. By substituting this explicit $K_{\mathtt{dec}}$ into the preceding Corollary~\ref{cor:monotone}, we obtain the following result. Consider $G(t)=\mu\cdot t^p$ for $p\ge 2$. The case of $p=2$ occurs when, for example, whenever $f$ is strongly convex. 

\begin{corollary}
\label{cor:accel}
Consider any $0<\epsilon< \Delta_0$. Suppose $f$ is $L$-smooth and convex, $\mathtt{fom}(\epsilon)$ is the accelerated method, Assumption~\ref{assumption:function-growth} holds for $f$ with $G(t) = \mu\cdot t^p$ for some $\mu>0$ and $p\ge 2$. Set $N_0=1$. 
(i) For $p>2$, set $\epsilon_k = \frac{\epsilon}{2e}\cdot \exp(c^k)$ for any $c\in (1, \frac{3}{2}-\frac{1}{p})$, then Algorithm~\ref{algo:general-framework} computes an $\epsilon$-optimal solution with iteration complexity and oracle complexity respectively at most 
\begin{equation}
\label{eq:cor-accel-eq1}
\mathcal{O}\left( \frac{\sqrt{L}}{ \mu^{\frac{1}{p}}} \cdot \frac{1}{\epsilon^{\frac{1}{2}-\frac{1}{p}}} \right) \quad\text{and}\quad \mathcal{O}\left( \frac{\sqrt{L}}{ \mu^{\frac{1}{p}}} \cdot \frac{1}{\epsilon^{\frac{1}{2}-\frac{1}{p}}} \cdot \log\log\left( \frac{\Delta_0}{\epsilon} \right) \right) .
\end{equation}
(ii) For $p=2$, set $\epsilon_k = \frac{\epsilon}{2}\cdot c^k$ for any $c>1$, then Algorithm~\ref{algo:general-framework} computes an $\epsilon$-optimal solution with iteration complexity and oracle complexity respectively at most 
\begin{equation*}
\mathcal{O}\left( \sqrt{\frac{L}{\mu}} \cdot \log\left( \frac{\Delta_0}{\epsilon} \right) \right) \quad\text{and}\quad \mathcal{O}\left( \sqrt{\frac{L}{\mu}} \cdot \left( \log\left( \frac{\Delta_0}{\epsilon} \right)\right)^2 \right) .
\end{equation*}
In either case, no other choice of $\{\epsilon_k\}$ can asymptotically improve these rates in~\eqref{eq:cor-monotone-eq2}, even the $\log$ and $\log\log$ factors.
\end{corollary}

\begin{proof}[Proof of Corollary~\ref{cor:accel}]
By Proposition~\ref{prop:accel-method}, Assumption~\ref{assumption:decrement-for-fom} holds with $K_{\mathtt{dec}}(\epsilon, \delta, D) = \sqrt{\frac{2LD^2}{\delta-\epsilon}}$. Then
\begin{equation*}
K_{\mathtt{dec}}(\epsilon, \delta, G^{-1}(\delta)) = \sqrt{ \frac{2L}{\delta-\epsilon} \cdot \left( \frac{\delta}{\mu} \right)^{\frac{2}{p}} } = \frac{\sqrt{2L}}{ \mu^{\frac{1}{p}}} \cdot \sqrt{\frac{\delta^{\frac{2}{p}} }{ \delta-\epsilon} } .
\end{equation*}
For any $0< r\le 1$, and any $\delta >\epsilon >0$, observe the partial derivative formula below
\begin{equation*}
\frac{\partial}{\partial \delta} \left( \frac{\delta^r}{ \delta-\epsilon} \right) = \frac{r\delta^{r-1} (\delta-\epsilon) - \delta^r}{ (\delta-\epsilon)^2} = \frac{\delta^{r-1}}{(\delta-\epsilon)^2} \left( (r-1)\delta - r\epsilon \right) \le 0 .
\end{equation*}
So $K_{\mathtt{dec}}(\epsilon, \delta, G^{-1}(\delta))$ is nonincreasing in $\delta$ for any $\delta>\epsilon>0$ and $p\ge 2$, which implies~\eqref{eq:cor-monotone-eq1}. By Corollary~\ref{cor:monotone}, if $\epsilon_0 \le \frac{\epsilon}{2}$, then Algorithm~\ref{algo:general-framework} computes an $\epsilon$-optimal solution with iteration complexity at most
\begin{equation}
\label{eq:proof-cor-accel-eq1}
T_n\le \sum_{k=0}^{m-1} \left( 1+ \frac{2\epsilon_{k+1}}{\epsilon_{k}} \right) K_{\mathtt{dec}} \left( \epsilon_k, 2 \epsilon_{k}, G^{-1}(2 \epsilon_{k}) \right) = \sum_{k=0}^{m-1} \left( 1+ \frac{2\epsilon_{k+1}}{\epsilon_{k}} \right) \frac{\sqrt{2L}}{ \mu^{\frac{1}{p}}} \cdot \sqrt{\frac{ (2\epsilon_k)^{\frac{2}{p}} }{ \epsilon_k} } .
\end{equation}
By Lemma~\ref{lemma:number-of-processes}, the number of processes in the scheme is at most $\hat{m}+1$ since $N_0=1$. So the total oracle complexity is at most $(\hat{m}+1) T_n$.

Since $\epsilon_k < \epsilon_{k+1}$, we have $1+ \frac{2\epsilon_{k+1}}{\epsilon_{k}} \le \frac{3\epsilon_{k+1}}{\epsilon_{k}}$. Then~\eqref{eq:proof-cor-accel-eq1} implies 
\begin{equation}
\label{eq:proof-cor-accel-eq2}
T_n = \mathcal{O}\left( \frac{\sqrt{L}}{ \mu^{\frac{1}{p}}} \sum_{k=0}^{m-1} \frac{\epsilon_{k+1}}{\epsilon_{k}^{\frac{3}{2}-\frac{1}{p}} } \right) \quad\text{and}\quad (\hat{m}+1) T_n= \mathcal{O}\left( (\hat{m}+1) \cdot \frac{\sqrt{L}}{ \mu^{\frac{1}{p}}} \sum_{k=0}^{m-1} \frac{\epsilon_{k+1}}{\epsilon_{k}^{\frac{3}{2}-\frac{1}{p}} } \right) .
\end{equation}
For $p>2$, by Proposition~\ref{prop:sublinear-rate}, if one sets $\epsilon_k = \frac{\epsilon}{2e}\cdot \exp(c^k)$ for any $c\in (1, \frac{3}{2}-\frac{1}{p})$, then the iteration complexity and oracle complexity are respectively at most
\begin{equation}
\label{eq:proof-cor-accel-eq3}
T_n= \mathcal{O}\left( \frac{\sqrt{L}}{ \mu^{\frac{1}{p}}} \cdot \frac{1}{\epsilon^{\frac{1}{2}-\frac{1}{p}}} \right) \quad\text{and}\quad (\hat{m}+1) T_n= \mathcal{O}\left( \frac{\sqrt{L}}{ \mu^{\frac{1}{p}}} \cdot \frac{1}{\epsilon^{\frac{1}{2}-\frac{1}{p}}} \cdot \log\log\left( \frac{\Delta_0}{\epsilon} \right) \right) .
\end{equation}
Proposition~\ref{prop:sublinear-rate} also guarantees that no other choice of $\{\epsilon_k\}$ can asymptotically improve the two rates in~\eqref{eq:proof-cor-accel-eq2} beyond~\eqref{eq:proof-cor-accel-eq3} if $p>2$.

We have similar results for $p=2$. For $p=2$, recall the complexity bounds in~\eqref{eq:proof-cor-accel-eq2}. If one sets $\epsilon_k = \frac{\epsilon}{2}\cdot c^k$ for any $c>1$, then by Proposition~\ref{prop:linear-rate}, the iteration complexity and oracle complexity are respectively at most
\begin{equation}
\label{eq:proof-cor-accel-eq4}
T_n = \mathcal{O}\left( \sqrt{\frac{L}{\mu}} \cdot \log\left( \frac{\Delta_0}{\epsilon} \right) \right) \quad\text{and}\quad (\hat{m}+1) T_n= \mathcal{O}\left( \sqrt{\frac{L}{\mu}} \cdot \left( \log\left( \frac{\Delta_0}{\epsilon} \right)\right)^2 \right) ,
\end{equation}
and no other choice of $\{\epsilon_k\}$ can asymptotically improve the two rates in~\eqref{eq:proof-cor-accel-eq2} beyond~\eqref{eq:proof-cor-accel-eq4} in this case.
\end{proof}

For $p=2$, the linear rate $\mathcal{O}\left( \sqrt{\frac{L}{\mu}} \cdot \log( \frac{\Delta_0}{\epsilon} ) \right)$ matches the well-known lower bound for strongly convex and smooth functions~\cite{NemirovskiYudin1983}. For $p>2$, the sublinear rate $\mathcal{O}\left( \frac{\sqrt{L}}{ \mu^{\frac{1}{p}}} \cdot \frac{1}{\epsilon^{\frac{1}{2}-\frac{1}{p}}} \right)$ also matches the lower bound stated by~\cite{nemirovskii1985}. For the total first-order oracle complexity, the parallel-process setting leads to an extra factor of $\mathcal{O}\left(\log\log(\frac{\Delta_0}{\epsilon})\right)$ for the sublinear rate case and $\mathcal{O}\left(\log(\frac{\Delta_0}{\epsilon})\right)$ for the linear rate case.

As discussed earlier, our Corollary above is also comparable with the restarting scheme in the prior work under the same $L$-smooth problem and the accelerated method, see~\cite[Corollary 5.2]{Renegar2022restart}, where they required the following iteration complexity and oracle complexity for the case $p>2$
\begin{equation*}
\mathcal{O}\left( \sqrt{L} \operatorname{dist}(x^{(0)}, X_\star) + \frac{\sqrt{L}}{ \mu^{\frac{1}{p}}} \cdot \frac{1}{\epsilon^{\frac{1}{2}-\frac{1}{p}}} \right) \quad\text{and}\quad \mathcal{O}\left( \left(\sqrt{L} \operatorname{dist}(x^{(0)}, X_\star) + \frac{\sqrt{L}}{ \mu^{\frac{1}{p}}} \cdot \frac{1}{\epsilon^{\frac{1}{2}-\frac{1}{p}}} \right) \cdot \log\left( \frac{1}{\epsilon} \right) \right).
\end{equation*}
Compared with~\eqref{eq:cor-accel-eq1}, our results maintain the same asymptotic term $\mathcal{O}\left( \frac{\sqrt{L}}{ \mu^{\frac{1}{p}}} \cdot \frac{1}{\epsilon^{\frac{1}{2}-\frac{1}{p}}} \right)$, while we remove the burn-in term $\mathcal{O}\left( \sqrt{L} \operatorname{dist}(x^{(0)}, X_\star) \right)$ and reduce the required number of processes to $\mathcal{O}(\log\log(1/\epsilon))$. Similar improvements are also achieved in the case $p=2$.

\begin{remark}
The result of Corollary~\ref{cor:accel} can be extended by taking $\mathtt{fom}$ to be FISTA~\cite{Beck2009fista}. FISTA applies to a more general composite form $\min_x f(x) = f_1(x)+f_2(x)$ where $f_1$ is $L$-smooth and convex, while $f_2$ is convex with a computable proximal operator. Its update differs from~\eqref{eq:accel} only in replacing the gradient step by $x^{(k+1)} = \operatorname{prox}_{\frac{f_2}{L}}\left( y^{(k)}- \frac{1}{L} \nabla f_1(y^{(k)}) \right)$, which reduces to a gradient descent step if $f_2=0$. If we assume more generally that $f$ has such a composite form, then the proof approach and results of Corollary~\ref{cor:accel} remain valid.
\end{remark}

\subsection{Application to the Smoothable Convex Setting}
\label{Subsec:smooth-monotone}

In nonsmooth optimization, a common technique is to approximate the nonsmooth objective function $f$ with a smooth function in order to apply efficient first-order methods. Specifically, given a possibly nonsmooth $f$, one can construct a smooth function $f_\eta$ that is uniformly close to $f$ on a domain of interest. Accelerated methods can then be applied to $f_\eta$, yielding improved convergence guarantees that rely on smoothness. By the uniform approximation between $f$ and $f_\eta$, a near-optimal solution of $f_\eta$ can be shown to be near-optimal for $f$, allowing the guarantees for the smoothed problem to transfer back to the original nonsmooth problem.

As an application of the general result in Corollary~\ref{cor:monotone}, we consider the following specific smoothable setting. Define an ``$(\alpha,\beta)$-convex smoothing family of $f$'' as a collection of functions $f_\eta$ parameterized by $\eta>0$, such that each $f_\eta$ is smooth, convex, and satisfies the following for any $x,y$,
\begin{equation*}
\|\nabla f_\eta(x) - \nabla f_\eta(y)\|_2 \le \frac{\alpha}{\eta} \|x-y\|_2, \quad\text{and}\quad f(x)\le f_\eta(x) \le f(x)+ \beta\eta .
\end{equation*}
Given an $(\alpha,\beta)$-convex smoothing family of $f$ with $\alpha,\beta>0$, we define the ``smoothing method'' as follows. Let $\mathtt{fom}(\epsilon)$ be the accelerated method applied to $f_\eta$ with $\eta = \frac{\epsilon}{4\beta}$. For this smoothing method, we have $K_{\mathtt{dec}}(\epsilon, \delta, D) = \sqrt{\frac{16\alpha\beta D^2}{\epsilon(\delta-\frac{3}{2}\epsilon)}}$, see Proposition~\ref{prop:smooth-method}. In this example, the monotonicity condition~\eqref{eq:cor-monotone-eq1} holds for $p\ge 2$, and Corollary~\ref{cor:monotone} gives the following result. 

\begin{corollary}
\label{cor:smooth-monotone}
Consider any $0<\epsilon< \Delta_0$. Suppose an $(\alpha,\beta)$-convex smoothing family of $f$ is known, $\mathtt{fom}(\epsilon)$ is the smoothing method, Assumption~\ref{assumption:function-growth} holds for $f$ with $G(t) = \mu\cdot t^p$ for some $\mu>0$ and $p\ge 2$. Set $N_0=1$ and $\epsilon_k = \frac{\epsilon}{2e}\cdot \exp(c^k)$ for any $c\in (1, 2-\frac{1}{p})$, then Algorithm~\ref{algo:general-framework} computes an $\epsilon$-optimal solution with iteration complexity and oracle complexity respectively at most 
\begin{equation*}
\mathcal{O}\left( \frac{ \sqrt{\alpha\beta}}{ \mu^{\frac{1}{p}}} \cdot \frac{1}{\epsilon^{1-\frac{1}{p}}} \right) \quad\text{and}\quad \mathcal{O}\left( \frac{ \sqrt{\alpha\beta}}{ \mu^{\frac{1}{p}}}\cdot \frac{1}{\epsilon^{1-\frac{1}{p}}} \cdot \log\log\left( \frac{\Delta_0}{\epsilon} \right) \right) .
\end{equation*}
No other choice of $\{\epsilon_k\}$ can asymptotically improve these rates in~\eqref{eq:cor-monotone-eq2}, even the $\log\log$ factor.
\end{corollary}

\begin{proof}[Proof of Corollary~\ref{cor:smooth-monotone}]
The proof is similar to the proof of Corollary~\ref{cor:accel}. By Proposition~\ref{prop:smooth-method}, Assumption~\ref{assumption:decrement-for-fom} holds with $K_{\mathtt{dec}}(\epsilon, \delta, D) = \sqrt{\frac{16\alpha\beta D^2}{\epsilon(\delta-\frac{3}{2}\epsilon)}}$. Then
\begin{equation}
\label{eq:proof-cor-smooth-monotone-eq1}
K_{\mathtt{dec}}(\epsilon, \delta, G^{-1}(\delta)) = \sqrt{ \frac{16\alpha\beta}{\epsilon (\delta -\frac{3}{2}\epsilon)} \cdot \left( \frac{\delta}{\mu} \right)^{\frac{2}{p}} } = \sqrt{ \frac{16\alpha\beta}{ \mu^{\frac{2}{p}}} } \cdot \sqrt{ \frac{\delta^{\frac{2}{p}}}{ \epsilon (\delta -\frac{3}{2}\epsilon)} }.
\end{equation}
For $\delta> \frac{3}{2}\epsilon$, observe the partial derivative formula below of
\begin{equation}
\label{eq:proof-cor-smooth-monotone-eq2}
\frac{\partial}{\partial \delta} \left( \frac{\delta^{\frac{2}{p}} }{\delta -\frac{3}{2}\epsilon} \right) = \frac{\frac{2}{p} \delta^{\frac{2}{p}-1} (\delta -\frac{3}{2}\epsilon) - \delta^{\frac{2}{p}} }{ (\delta -\frac{3}{2}\epsilon)^2 } = \frac{ \delta^{\frac{2}{p}-1}}{ p (\delta -\frac{3}{2}\epsilon)^2} \left( (2-p)\delta- 3\epsilon \right) .
\end{equation}
So $K_{\mathtt{dec}}(\epsilon, \delta, G^{-1}(\delta))$ is nonincreasing in $\delta$ for any $\delta> \frac{3}{2}\epsilon$ and $p\ge 2$, which implies~\eqref{eq:cor-monotone-eq1}. By Corollary~\ref{cor:monotone}, if $\epsilon_0 \le \frac{\epsilon}{2}$, then Algorithm~\ref{algo:general-framework} computes an $\epsilon$-optimal solution with iteration complexity at most
\begin{equation}
\label{eq:proof-cor-smooth-monotone-eq3}
T_n \le \sum_{k=0}^{m-1} \left( 1+ \frac{2\epsilon_{k+1}}{\epsilon_{k}} \right) K_{\mathtt{dec}} \left( \epsilon_k, 2 \epsilon_{k}, G^{-1}(2 \epsilon_{k}) \right) = \mathcal{O}\left( \frac{ \sqrt{\alpha\beta}}{ \mu^{\frac{1}{p}}} \sum_{k=0}^{m-1} \frac{\epsilon_{k+1}}{\epsilon_{k}^{2-\frac{1}{p}} } \right) .
\end{equation}
By Lemma~\ref{lemma:number-of-processes}, the number of processes in the scheme is at most $\hat{m}+1$. So the total oracle complexity is at most
\begin{equation}
\label{eq:proof-cor-smooth-monotone-eq4}
(\hat{m}+1) T_n = \mathcal{O}\left( (\hat{m}+1)\cdot \frac{ \sqrt{\alpha\beta}}{ \mu^{\frac{1}{p}}} \sum_{k=0}^{m-1} \frac{\epsilon_{k+1}}{\epsilon_{k}^{2-\frac{1}{p}} } \right) .
\end{equation}
By Proposition~\ref{prop:sublinear-rate}, if one sets $\epsilon_k = \frac{\epsilon}{2e}\cdot \exp(c^k)$ for any $c\in (1, 2-\frac{1}{p})$, then the iteration complexity and oracle complexity are respectively at most
\begin{equation}
\label{eq:proof-cor-smooth-monotone-eq5}
T_n= \mathcal{O}\left( \frac{ \sqrt{\alpha\beta}}{ \mu^{\frac{1}{p}}} \cdot \frac{1}{\epsilon^{1-\frac{1}{p}}} \right) \quad\text{and}\quad (\hat{m}+1)T_n = \mathcal{O}\left( \frac{ \sqrt{\alpha\beta}}{ \mu^{\frac{1}{p}}}\cdot \frac{1}{\epsilon^{1-\frac{1}{p}}} \cdot \log\log\left( \frac{\Delta_0}{\epsilon} \right) \right) ,
\end{equation}
and no choice of $\{\epsilon_k\}$ can asymptotically improve the two rates in~\eqref{eq:proof-cor-smooth-monotone-eq3} and~\eqref{eq:proof-cor-smooth-monotone-eq4} beyond~\eqref{eq:proof-cor-smooth-monotone-eq5}.
\end{proof}

Similar to the previous example, our Corollary~\ref{cor:smooth-monotone} improves upon the rates in~\cite{Renegar2022restart} under the same setting, where they required an additional burn-in term $\mathcal{O}\left( \sqrt{\alpha\beta} \operatorname{dist}(x^{(0)}, X_\star) \right)$ and a larger number of processes $\mathcal{O}(\log(1/\epsilon))$.

From~\eqref{eq:proof-cor-smooth-monotone-eq1} and~\eqref{eq:proof-cor-smooth-monotone-eq2} in the proof, we see that the monotonicity condition~\eqref{eq:cor-monotone-eq1} does not hold for $p<2$, so Corollary~\ref{cor:monotone} cannot be applied. Nevertheless, the case $p<2$ in this example can still be handled by Theorem~\ref{thm:main}, developed in Section~\ref{Sect:cor-non-monotone}.

\subsection{Application to the Lipschitz Convex Setting}
\label{Subsec:subgrad-monotone}

For the class of $M$-Lipschitz convex functions, let $\mathtt{fom}(\epsilon)$ be the following subgradient method parameterized by $\epsilon >0$. Let $x^{(i)}$ be the $i$-th iterate and $g^{(i)} \in \partial f(x^{(i)})$. If $g^{(i)}=0$, one can terminate the method and declare $x^{(i)}$ as a minimizer. If $g^{(i)} \ne 0$, let $x^{(i+1)} = x^{(i)} - \frac{\epsilon}{\|g^{(i)}\|_2^2} g^{(i)}$.

For this subgradient method, we have the decrement guarantee $K_{\mathtt{dec}}(\epsilon, \delta, D) = \frac{M^2 D^2}{2\epsilon (\delta -\frac{3}{2}\epsilon)}$, see Proposition~\ref{prop:subgrad-method}. As in the previous two examples, we still consider $G(t)=\mu\cdot t^p$. In this example, the monotonicity condition~\eqref{eq:cor-monotone-eq1} holds for $p\ge 2$, so Corollary~\ref{cor:monotone} still applies.

\begin{corollary}
\label{cor:subgrad-monotone}
Consider any $0<\epsilon< \Delta_0$. Suppose $f$ is $M$-Lipschitz and convex, $\mathtt{fom}(\epsilon)$ is the subgradient method, Assumption~\ref{assumption:function-growth} holds for $f$ with $G(t) = \mu\cdot t^p$ for some $\mu>0$ and $p\ge 2$. Set $N_0=1$ and $\epsilon_k = \frac{\epsilon}{2e}\cdot \exp(c^k)$ for any $c\in (1, 3-\frac{2}{p})$, then Algorithm~\ref{algo:general-framework} computes an $\epsilon$-optimal solution with iteration complexity and oracle complexity respectively at most 
\begin{equation*}
\mathcal{O}\left( \frac{M^2}{\mu^{\frac{2}{p}} } \cdot \frac{1}{\epsilon^{2-\frac{2}{p}}} \right) \quad\text{and}\quad \mathcal{O}\left( \frac{M^2}{\mu^{\frac{2}{p}} } \cdot \frac{1}{\epsilon^{2-\frac{2}{p}}} \cdot \log\log\left( \frac{\Delta_0}{\epsilon} \right) \right) .
\end{equation*}
No other choice of $\{\epsilon_k\}$ can asymptotically improve these rates in~\eqref{eq:cor-monotone-eq2}, even the $\log\log$ factor.
\end{corollary}

\begin{proof}[Proof of Corollary~\ref{cor:subgrad-monotone}]
The proof parallels that of Corollary~\ref{cor:smooth-monotone}. By Proposition~\ref{prop:subgrad-method}, Assumption~\ref{assumption:decrement-for-fom} holds with $K_{\mathtt{dec}}(\epsilon, \delta, D) = \frac{M^2 D^2}{2\epsilon (\delta -\frac{3}{2}\epsilon)}$. Then
\begin{equation}
\label{eq:proof-cor-subgrad-monotone-eq1}
K_{\mathtt{dec}}(\epsilon, \delta, G^{-1}(\delta)) = \frac{M^2}{2\epsilon (\delta -\frac{3}{2}\epsilon)} \cdot \left( \frac{\delta}{\mu} \right)^{\frac{2}{p}} = \frac{M^2}{2 \mu^{\frac{2}{p}}} \cdot \frac{\delta^{\frac{2}{p}}}{ \epsilon (\delta -\frac{3}{2}\epsilon)} .
\end{equation}
By~\eqref{eq:proof-cor-smooth-monotone-eq2}, $K_{\mathtt{dec}}(\epsilon, \delta, G^{-1}(\delta))$ is nonincreasing in $\delta$ for any $\delta> \frac{3}{2}\epsilon$ and $p\ge 2$, which implies~\eqref{eq:cor-monotone-eq1}. By Corollary~\ref{cor:monotone}, if $\epsilon_0 \le \frac{\epsilon}{2}$, then Algorithm~\ref{algo:general-framework} computes an $\epsilon$-optimal solution with iteration complexity at most
\begin{equation}
\label{eq:proof-cor-subgrad-monotone-eq3}
T_n \le \sum_{k=0}^{m-1} \left( 1+ \frac{2\epsilon_{k+1}}{\epsilon_{k}} \right) K_{\mathtt{dec}} \left( \epsilon_k, 2 \epsilon_{k}, G^{-1}(2 \epsilon_{k}) \right) = \mathcal{O}\left( \frac{M^2}{\mu^{\frac{2}{p}} } \sum_{k=0}^{m-1} \frac{\epsilon_{k+1}}{\epsilon_{k}^{3-\frac{2}{p}} } \right) .
\end{equation}
Using Lemma~\ref{lemma:number-of-processes}, the total oracle complexity is at most
\begin{equation}
\label{eq:proof-cor-subgrad-monotone-eq4}
(\hat{m}+1) T_n = \mathcal{O}\left( (\hat{m}+1) \cdot \frac{M^2}{\mu^{\frac{2}{p}} } \sum_{k=0}^{m-1} \frac{\epsilon_{k+1}}{\epsilon_{k}^{3-\frac{2}{p}} } \right) .
\end{equation}
By Proposition~\ref{prop:sublinear-rate}, if one sets $\epsilon_k = \frac{\epsilon}{2e}\cdot \exp(c^k)$ for any $c\in (1, 3-\frac{2}{p})$, then the iteration complexity and oracle complexity are respectively at most
\begin{equation}
\label{eq:proof-cor-subgrad-monotone-eq5}
T_n= \mathcal{O}\left( \frac{M^2}{\mu^{\frac{2}{p}} } \cdot \frac{1}{\epsilon^{2-\frac{2}{p}}} \right) \quad\text{and}\quad (\hat{m}+1)T_n = \mathcal{O}\left( \frac{M^2}{\mu^{\frac{2}{p}} } \cdot \frac{1}{\epsilon^{2-\frac{2}{p}}} \cdot \log\log\left( \frac{\Delta_0}{\epsilon} \right) \right) ,
\end{equation}
and no choice of $\{\epsilon_k\}$ can asymptotically improve the two rates in~\eqref{eq:proof-cor-subgrad-monotone-eq3} and~\eqref{eq:proof-cor-subgrad-monotone-eq4} beyond~\eqref{eq:proof-cor-subgrad-monotone-eq5}.
\end{proof}

\section{Applications under Growth with $p\in [1,2)$}
\label{Sect:cor-non-monotone}

For the previous two examples in Section~\ref{Subsec:smooth-monotone} and Section~\ref{Subsec:subgrad-monotone}, the monotonicity property~\eqref{eq:cor-monotone-eq1} and Corollary~\ref{cor:monotone} fail for the case $p\in [1,2)$. In this section, we consider an alternative setting to that of Corollary~\ref{cor:monotone}, under which Theorem~\ref{thm:main} can still be simplified and applied to the previous examples when $p\in [1,2)$. This allows us to get usable bounds for $p\in[1,2)$, at the cost of requiring a (loose) upper bound on $\Delta_0$. 

Note that $N_0$ can be chosen differently from the default value $1$ in Algorithm~\ref{algo:general-framework}. If $N_0>m$, then Phase 1 in Theorem~\ref{thm:main} becomes $0$. By enlarging the intervals $[\delta_\text{min}, \delta_\text{max}]$ in the terms $Q_k\left( \delta_\text{min}, \delta_\text{max} \right)$, Phase 2 in Theorem~\ref{thm:main} can be further upper bounded by a much simpler form. This leads to the following guarantee whenever $N_0>m$, or equivalently, $\Delta_0 \le 2\epsilon_{N_0-1}$.

\begin{corollary}
\label{cor:non-monotone-cor}
Suppose Assumption~\ref{assumption:decrement-for-fom} holds for $\mathtt{fom}(\epsilon)$ and Assumption~\ref{assumption:function-growth} holds for $f$. If $\Delta_0 \le 2\epsilon_{N_0-1}$, then Algorithm~\ref{algo:general-framework} computes a $(2\epsilon_0)$-optimal solution no later than round
\begin{equation}
\label{eq:cor-non-monotone-eq1}
T_n\le \sum_{k=0}^{m-1} \frac{2\epsilon_{k+1}}{\epsilon_{k}} \cdot Q_k\left( 2 \epsilon_{k}, 3\epsilon_{k+1} \right) .
\end{equation}
\end{corollary}

\begin{proof}[Proof of Corollary~\ref{cor:non-monotone-cor}]
Since $\Delta_0 \le 2\epsilon_{N_0-1}$, we have $m\le N_0-1<N_0$ by the definition of $m$ in~\eqref{eq:m}. Then the Phase 1 term in Theorem~\ref{thm:main} is zero, so Algorithm~\ref{algo:general-framework} computes a $(2\epsilon_0)$-optimal solution no later than round
\begin{equation}
\label{eq:proof-cor-non-monotone-cor-eq1}
T_n \le \sum_{k=0}^{m-1}\ \max_{1\le r < \frac{2\epsilon_{k+1}}{\epsilon_{k}} }\ \sum_{i=0}^{r-1} Q_k\left( (r+1-i) \epsilon_{k}, \min \{ 2\epsilon_{k+1} + (1-i) \epsilon_{k}, \Delta_0 \} \right) .
\end{equation}
For each $0\le i\le r-1$, we have the inequalities $(r+1-i) \epsilon_{k} \ge 2\epsilon_{k}$, and $\min \{ 2\epsilon_{k+1} + (1-i) \epsilon_{k}, \Delta_0 \} \le 2\epsilon_{k+1} + (1-i) \epsilon_{k} \le 3\epsilon_{k+1}$. Then by the definition of $Q_k$ in~\eqref{eq:Q},
\begin{equation}
\label{eq:proof-cor-non-monotone-cor-eq2}
Q_k\left( (r+1-i) \epsilon_{k}, \min \{ 2\epsilon_{k+1} + (1-i) \epsilon_{k}, \Delta_0 \} \right) \le Q_k\left( 2 \epsilon_{k}, 3\epsilon_{k+1} \right).
\end{equation}
Using~\eqref{eq:proof-cor-non-monotone-cor-eq2} to control the terms in~\eqref{eq:proof-cor-non-monotone-cor-eq1}, we obtain
\begin{equation*}
T_n \le \sum_{k=0}^{m-1}\ \max_{1\le r < \frac{2\epsilon_{k+1}}{\epsilon_{k}} }\ \sum_{i=0}^{r-1} Q_k\left( 2 \epsilon_{k}, 3\epsilon_{k+1} \right) \le \sum_{k=0}^{m-1} \frac{2\epsilon_{k+1}}{\epsilon_{k}} \cdot Q_k\left( 2 \epsilon_{k}, 3\epsilon_{k+1} \right) . \qedhere
\end{equation*}
\end{proof}

It is worth noting that, unlike the results in Corollary~\ref{cor:monotone} and Section~\ref{Sect:cor-monotone}, the assumption $\Delta_0\le 2\epsilon_{N_0-1}$ requires an upper bound on $\Delta_0$ to choose $N_0$. So Corollary~\ref{cor:non-monotone-cor} is not fully parameter-free. The dependence of $N_0$ on $\Delta_0$ is illustrated in the corollaries below, where we apply Corollary~\ref{cor:non-monotone-cor} to the case $p\in [1,2)$ of the previous two examples.

We first consider the smoothing method discussed in Section~\ref{Subsec:smooth-monotone}. The following analysis and results for iteration complexity are similar to the corollaries in Section~\ref{Sect:cor-monotone}. For oracle complexities, the multiplicative factor $\max\{\hat{m}+1, N_0\}$ now strongly depends on $N_0$, which is no longer parameter-free. Nevertheless, the result below shows that the required $N_0$ is not too large if the sequence $\{\epsilon_k\}$ grows moderately fast. In particular, a doubly exponential sequence $\{\epsilon_k\}$ gives a sublinear iteration complexity (which is also the optimal rate one can obtain from Corollary~\ref{cor:non-monotone-cor}), and the corresponding $N_0= \Omega\left( \log\log\left( \frac{\Delta_0}{\epsilon} \right) \right)$ only requires an upper bound for a doubly logarithmic term.

\begin{corollary}
\label{cor:smooth-non-monotone}
Consider any $0<\epsilon< \Delta_0$. Suppose an $(\alpha,\beta)$-convex smoothing family of $f$ is known, $\mathtt{fom}(\epsilon)$ is the smoothing method, Assumption~\ref{assumption:function-growth} holds for $f$ with $G(t) = \mu\cdot t^p$ for some $\mu>0$ and $1\le p<2$. Assume
\begin{equation}
\label{eq:cor-smooth-non-monotone-eq1}
\text{$N_0$ is chosen so that } \Delta_0 \le 2\epsilon_{N_0-1}.
\end{equation}
(i) For $1<p<2$, set $\epsilon_k = \frac{\epsilon}{2e}\cdot \exp(c^k)$ for any $c\in (1, \frac{3p}{p+2})$, $N_0= \Omega\left( \log\log\left( \frac{\Delta_0}{\epsilon} \right) \right)$, then Algorithm~\ref{algo:general-framework} computes an $\epsilon$-optimal solution with iteration complexity and oracle complexity respectively at most
\begin{equation*}
\mathcal{O}\left( \frac{ \sqrt{\alpha\beta}}{ \mu^{\frac{1}{p}}} \cdot \frac{1}{\epsilon^{1-\frac{1}{p}}} \right) \quad\text{and}\quad \mathcal{O}\left( \frac{ \sqrt{\alpha\beta}}{ \mu^{\frac{1}{p}}}\cdot \frac{1}{\epsilon^{1-\frac{1}{p}}} \cdot N_0 \right) .
\end{equation*}
(ii) For $p=1$, set $\epsilon_k = \frac{\epsilon}{2}\cdot c^k$ for any $c>1$, $N_0= \Omega\left( \log\left( \frac{\Delta_0}{\epsilon} \right) \right)$, then Algorithm~\ref{algo:general-framework} computes an $\epsilon$-optimal solution with iteration complexity and oracle complexity respectively at most
\begin{equation*}
\mathcal{O}\left( \frac{ \sqrt{\alpha\beta}}{ \mu^{\frac{1}{p}}} \cdot \log\left( \frac{\Delta_0}{\epsilon} \right) \right) \quad\text{and}\quad \mathcal{O}\left( \frac{ \sqrt{\alpha\beta}}{ \mu^{\frac{1}{p}}} \cdot \log\left( \frac{\Delta_0}{\epsilon} \right) \cdot N_0 \right) .
\end{equation*}
In either case, no other choice of $\{\epsilon_k\}$ can asymptotically improve these rates in~\eqref{eq:cor-non-monotone-eq1} under~\eqref{eq:cor-smooth-non-monotone-eq1}.
\end{corollary}

\begin{proof}[Proof of Corollary~\ref{cor:smooth-non-monotone}]
Given the assumption $\Delta_0 \le 2\epsilon_{N_0-1}$, by Corollary~\ref{cor:non-monotone-cor}, if $\epsilon_0 \le \frac{\epsilon}{2}$, then Algorithm~\ref{algo:general-framework} computes an $\epsilon$-optimal solution no later than round
\begin{equation}
\label{eq:proof-cor-smooth-non-monotone-eq1}
T_n\le \sum_{k=0}^{m-1} \frac{2\epsilon_{k+1}}{\epsilon_{k}} \cdot Q_k\left( 2 \epsilon_{k}, 3\epsilon_{k+1} \right) .
\end{equation}
Note $K_{\mathtt{dec}}(\epsilon, \delta, G^{-1}(\delta)) = \sqrt{ \frac{16\alpha\beta}{ \mu^{\frac{2}{p}}} \cdot \frac{\delta^{\frac{2}{p}}}{ \epsilon (\delta -\frac{3}{2}\epsilon)} }$, as given by~\eqref{eq:proof-cor-smooth-monotone-eq1}. Then by~\eqref{eq:proof-cor-smooth-monotone-eq2}, for any fixed $\epsilon>0$ and $1\le p<2$, $K_{\mathtt{dec}}(\epsilon, \delta, G^{-1}(\delta))$ is nonincreasing in $\delta$ on the interval $(\frac{3}{2}\epsilon, \frac{3}{2-p}\epsilon ]$, and nondecreasing on $[\frac{3}{2-p}\epsilon, +\infty)$. So
\begin{equation}
\label{eq:proof-cor-smooth-non-monotone-eq2}
\begin{split}
& Q_k\left( 2\epsilon_{k}, 3\epsilon_{k+1} \right) = \sup_{\delta \in [2\epsilon_{k}, 3\epsilon_{k+1} ]} K_{\mathtt{dec}} \left( \epsilon_{k}, \delta, G^{-1}(\delta) \right) \\
&\le \max\left\{ K_{\mathtt{dec}} \left( \epsilon_{k}, 2\epsilon_{k}, G^{-1}(2\epsilon_{k}) \right), K_{\mathtt{dec}} \left( \epsilon_{k}, 3\epsilon_{k+1}, G^{-1}(3\epsilon_{k+1}) \right) \right\} \\
&\le K_{\mathtt{dec}} \left( \epsilon_{k}, 2\epsilon_{k}, G^{-1}(2\epsilon_{k}) \right) + K_{\mathtt{dec}} \left( \epsilon_{k}, 3\epsilon_{k+1}, G^{-1}(3\epsilon_{k+1}) \right) \\
&= \sqrt{ \frac{16\alpha\beta}{ \mu^{\frac{2}{p}}} } \cdot \left( \sqrt{ \frac{(2\epsilon_{k})^{\frac{2}{p}}}{ \frac{1}{2} \epsilon_{k}^2}} + \sqrt{ \frac{(3\epsilon_{k+1})^{\frac{2}{p}}}{ \epsilon_k (3\epsilon_{k+1} -\frac{3}{2}\epsilon_k)}} \right).
\end{split}
\end{equation}
The rest of this proof follows similarly to that of Corollary~\ref{cor:accel}. Combining~\eqref{eq:proof-cor-smooth-non-monotone-eq1} and~\eqref{eq:proof-cor-smooth-non-monotone-eq2}, the iteration complexity is at most
\begin{equation}
\label{eq:proof-cor-smooth-non-monotone-eq3}
\begin{split}
T_n &\le \sqrt{ \frac{16\alpha\beta}{ \mu^{\frac{2}{p}}}} \cdot \left( \sum_{k=0}^{m-1} \frac{2\epsilon_{k+1}}{\epsilon_{k}} \cdot \sqrt{ \frac{ (2\epsilon_{k})^{\frac{2}{p}}}{ \frac{1}{2} \epsilon_{k}^2}} + \sum_{k=0}^{m-1} \frac{2\epsilon_{k+1}}{\epsilon_{k}} \cdot \sqrt{ \frac{ (3\epsilon_{k+1})^{\frac{2}{p}}}{ \epsilon_k (3\epsilon_{k+1} -\frac{3}{2}\epsilon_k)}} \right) \\
&= \mathcal{O}\left( \frac{ \sqrt{\alpha\beta}}{ \mu^{\frac{1}{p}}} \cdot \left( \sum_{k=0}^{m-1} \frac{\epsilon_{k+1}}{\epsilon_{k}^{2-\frac{1}{p}} } + \sum_{k=0}^{m-1} \frac{\epsilon_{k+1}^{\frac{1}{2}+ \frac{1}{p}} }{\epsilon_{k}^{\frac{3}{2}}} \right) \right) ,
\end{split}
\end{equation}
where the last step is because $\frac{1}{3\epsilon_{k+1} -\frac{3}{2}\epsilon_k} \le \frac{2}{3\epsilon_{k+1}}$. By Lemma~\ref{lemma:number-of-processes}, the number of processes in the scheme is at most $\max\{\hat{m}+1, N_0\}$. So the total oracle complexity is at most
\begin{equation}
\label{eq:proof-cor-smooth-non-monotone-eq4}
\max\{\hat{m}+1, N_0\} \cdot T_n = \mathcal{O}\left( \max\{\hat{m}+1, N_0\} \cdot \frac{ \sqrt{\alpha\beta}}{ \mu^{\frac{1}{p}}} \cdot \left( \sum_{k=0}^{m-1} \frac{\epsilon_{k+1}}{\epsilon_{k}^{2-\frac{1}{p}} } + \sum_{k=0}^{m-1} \frac{\epsilon_{k+1}^{\frac{1}{2}+ \frac{1}{p}} }{\epsilon_{k}^{\frac{3}{2}}} \right) \right) .
\end{equation}
For $1<p<2$, if one sets $\epsilon_k = \frac{\epsilon}{2e}\cdot \exp(c^k)$ for any $c\in (1, \frac{3p}{p+2})$, then $\hat{m}= \Theta\left( \log\log\left( \frac{\Delta_0}{\epsilon} \right) \right)$ and $N_0 = \Omega\left( \log\log\left( \frac{\Delta_0}{\epsilon} \right) \right)$, by the definition of $\hat{m}$ in~\eqref{eq:m} and the assumption for $N_0$ in~\eqref{eq:cor-smooth-non-monotone-eq1} respectively. So $\max\{ \hat{m}+1, N_0\} = \mathcal{O}(N_0)$. Consider applying Proposition~\ref{prop:sublinear-rate} to $(a_1,b_1) = (1, 2-\frac{1}{p})$ and $(a_2,b_2) = (\frac{1}{2}+\frac{1}{p}, \frac{3}{2})$, which requires $1<c< \frac{b_1}{a_1}= \frac{2p-1}{p}$ and $1<c< \frac{b_2}{a_2}= \frac{3p}{p+2}$. Here we have $\frac{3p}{p+2}\le \frac{2p-1}{p}$ since $1<p<2$. So setting $c\in (1, \frac{3p}{p+2})$ satisfies the requirements. Then by Proposition~\ref{prop:sublinear-rate}, for this choice of $\{\epsilon_k\}$, the iteration complexity and oracle complexity are respectively at most
\begin{equation}
\label{eq:proof-cor-smooth-non-monotone-eq5}
T_n = \mathcal{O}\left( \frac{ \sqrt{\alpha\beta}}{ \mu^{\frac{1}{p}}} \cdot \frac{1}{\epsilon^{1-\frac{1}{p}}} \right) \quad\text{and}\quad \max\{\hat{m}+1, N_0\}\cdot T_n = \mathcal{O}\left( \frac{ \sqrt{\alpha\beta}}{ \mu^{\frac{1}{p}}} \cdot \frac{1}{\epsilon^{1-\frac{1}{p}}} \cdot N_0 \right) .
\end{equation}
Consider the lower bounds for the two asymptotic rates in~\eqref{eq:proof-cor-smooth-non-monotone-eq3} and~\eqref{eq:proof-cor-smooth-non-monotone-eq4}. First, we have $\sum_{k=0}^{m-1} \frac{\epsilon_{k+1}}{\epsilon_{k}^{2-\frac{1}{p}} } = \Omega(\epsilon^{-1+\frac{1}{p}})$ from Proposition~\ref{prop:sublinear-rate}. Also, note that $\max\{ \hat{m}+1, N_0\} \ge N_0$. Combining these two facts, no choice of $\{\epsilon_k\}$ can asymptotically improve the two rates in~\eqref{eq:proof-cor-smooth-non-monotone-eq3} and~\eqref{eq:proof-cor-smooth-non-monotone-eq4} beyond~\eqref{eq:proof-cor-smooth-non-monotone-eq5} if $1<p<2$.

Similarly, for $p=1$, if one sets $\epsilon_k = \frac{\epsilon}{2}\cdot c^k$ for any $c>1$, then $\hat{m}= \Theta\left( \log\left( \frac{\Delta_0}{\epsilon} \right) \right)$ and $N_0 = \Omega\left( \log\left( \frac{\Delta_0}{\epsilon} \right) \right)$. So $\max\{ \hat{m}+1, N_0\} = \mathcal{O}(N_0)$. Then by Proposition~\ref{prop:linear-rate}, the iteration complexity and oracle complexity are respectively at most
\begin{equation}
\label{eq:proof-cor-smooth-non-monotone-eq6}
T_n= \mathcal{O}\left( \frac{ \sqrt{\alpha\beta}}{ \mu^{\frac{1}{p}}} \cdot \log\left( \frac{\Delta_0}{\epsilon} \right) \right) \quad\text{and}\quad \max\{ \hat{m}+1, N_0\} \cdot T_n = \mathcal{O}\left( \frac{ \sqrt{\alpha\beta}}{ \mu^{\frac{1}{p}}} \cdot \log\left( \frac{\Delta_0}{\epsilon} \right) \cdot N_0 \right) .
\end{equation}
By the lower bound in Proposition~\ref{prop:linear-rate} and the fact $\max\{ \hat{m}+1, N_0\} \ge N_0$, no choice of $\{\epsilon_k\}$ can asymptotically improve the two rates in~\eqref{eq:proof-cor-smooth-non-monotone-eq3} and~\eqref{eq:proof-cor-smooth-non-monotone-eq4} beyond~\eqref{eq:proof-cor-smooth-non-monotone-eq6} if $p=1$.
\end{proof}

Next, we consider the subgradient method and Lipschitz convex functions as described in Section~\ref{Subsec:subgrad-monotone}. As in the previous application, Corollary~\ref{cor:non-monotone-cor} handles the case $p\in [1,2)$ and leads to the following result.

\begin{corollary}
\label{cor:subgrad-non-monotone}
Consider any $0<\epsilon< \Delta_0$. Suppose $f$ is $M$-Lipschitz and convex, $\mathtt{fom}(\epsilon)$ is the subgradient method, Assumption~\ref{assumption:function-growth} holds for $f$ with $G(t) = \mu\cdot t^p$ for some $\mu>0$ and $1\le p< 2$. Assume 
\begin{equation}
\label{eq:cor-subgrad-non-monotone-eq1}
\text{$N_0$ is chosen so that } \Delta_0 \le 2\epsilon_{N_0-1}.
\end{equation}
(i) For $1<p<2$, set $\epsilon_k = \frac{\epsilon}{2e}\cdot \exp(c^k)$ for any $c\in (1, p)$, $N_0= \Omega\left( \log\log\left( \frac{\Delta_0}{\epsilon} \right) \right)$, then Algorithm~\ref{algo:general-framework} computes an $\epsilon$-optimal solution with iteration complexity and oracle complexity respectively at most
\begin{equation*}
\mathcal{O}\left( \frac{M^2}{ \mu^{\frac{2}{p}}} \cdot \frac{1}{\epsilon^{2-\frac{2}{p}}} \right) \quad\text{and}\quad \mathcal{O}\left( \frac{M^2}{\mu^{\frac{2}{p}} } \cdot \frac{1}{\epsilon^{2-\frac{2}{p}}} \cdot N_0 \right) .
\end{equation*}
(ii) For $p=1$, set $\epsilon_k = \frac{\epsilon}{2}\cdot c^k$ for any $c>1$, $N_0= \Omega\left( \log\left( \frac{\Delta_0}{\epsilon} \right) \right)$, then Algorithm~\ref{algo:general-framework} computes an $\epsilon$-optimal solution with iteration complexity and oracle complexity respectively at most
\begin{equation*}
\mathcal{O}\left( \frac{M^2}{ \mu^{\frac{2}{p}}} \cdot \log\left( \frac{\Delta_0}{\epsilon} \right) \right) \quad\text{and}\quad \mathcal{O}\left( \frac{M^2}{ \mu^{\frac{2}{p}}} \cdot \log\left( \frac{\Delta_0}{\epsilon} \right) \cdot N_0 \right) .
\end{equation*}
In either case, no other choice of $\{\epsilon_k\}$ can asymptotically improve these rates in~\eqref{eq:cor-non-monotone-eq1} under~\eqref{eq:cor-subgrad-non-monotone-eq1}.
\end{corollary}

\begin{proof}[Proof of Corollary~\ref{cor:subgrad-non-monotone}]
The proof is similar to the proof of Corollary~\ref{cor:smooth-non-monotone}. Given $\Delta_0 \le 2\epsilon_{N_0-1}$, by Corollary~\ref{cor:non-monotone-cor}, if $\epsilon_0 \le \frac{\epsilon}{2}$, then Algorithm~\ref{algo:general-framework} computes an $\epsilon$-optimal solution no later than round $T_n\le \sum_{k=0}^{m-1} \frac{2\epsilon_{k+1}}{\epsilon_{k}} \cdot Q_k\left( 2 \epsilon_{k}, 3\epsilon_{k+1} \right)$.

Note $K_{\mathtt{dec}}(\epsilon, \delta, G^{-1}(\delta)) = \frac{M^2}{2 \mu^{\frac{2}{p}}} \cdot \frac{\delta^{\frac{2}{p}}}{ \epsilon (\delta -\frac{3}{2}\epsilon)}$, as given by~\eqref{eq:proof-cor-subgrad-monotone-eq1}. Then by~\eqref{eq:proof-cor-smooth-monotone-eq2}, for any fixed $\epsilon>0$ and $1\le p<2$, $K_{\mathtt{dec}}(\epsilon, \delta, G^{-1}(\delta))$ is nonincreasing in $\delta$ on the interval $(\frac{3}{2}\epsilon, \frac{3}{2-p}\epsilon ]$, and nondecreasing on $[\frac{3}{2-p}\epsilon, +\infty)$. Using the same technique as in~\eqref{eq:proof-cor-smooth-non-monotone-eq2}, we have
\begin{equation*}
\begin{split}
Q_k\left( 2\epsilon_{k}, 3\epsilon_{k+1} \right) &\le K_{\mathtt{dec}} \left( \epsilon_{k}, 2\epsilon_{k}, G^{-1}(2\epsilon_{k}) \right) + K_{\mathtt{dec}} \left( \epsilon_{k}, 3\epsilon_{k+1}, G^{-1}(3\epsilon_{k+1}) \right) \\
&= \frac{M^2}{2 \mu^{\frac{2}{p}}} \cdot \left( \frac{(2\epsilon_k)^{\frac{2}{p}}}{ \frac{1}{2} \epsilon_k^2 } + \frac{(3\epsilon_{k+1})^{\frac{2}{p}}}{ \epsilon_k (3\epsilon_{k+1} -\frac{3}{2}\epsilon_k)} \right).
\end{split}
\end{equation*}
So the iteration complexity is at most
\begin{equation}
\label{eq:proof-cor-subgrad-non-monotone-eq3}
\begin{split}
T_n &\le \frac{M^2}{2 \mu^{\frac{2}{p}}} \cdot \left( \sum_{k=0}^{m-1} \frac{2\epsilon_{k+1}}{\epsilon_{k}} \cdot \frac{(2\epsilon_{k})^{\frac{2}{p}} }{ \frac{1}{2} \epsilon_{k}^2} + \sum_{k=0}^{m-1} \frac{2\epsilon_{k+1}}{\epsilon_{k}} \cdot \frac{(3\epsilon_{k+1})^{\frac{2}{p}}}{ \epsilon_k (3\epsilon_{k+1} -\frac{3}{2}\epsilon_k)} \right) \\
&= \mathcal{O}\left( \frac{M^2}{ \mu^{\frac{2}{p}}} \cdot \left( \sum_{k=0}^{m-1} \frac{\epsilon_{k+1}}{\epsilon_{k}^{3-\frac{2}{p}} } + \sum_{k=0}^{m-1} \frac{\epsilon_{k+1}^{\frac{2}{p}} }{\epsilon_{k}^2} \right) \right) .
\end{split}
\end{equation}
By Lemma~\ref{lemma:number-of-processes}, the total oracle complexity is at most
\begin{equation}
\label{eq:proof-cor-subgrad-non-monotone-eq4}
\max\{\hat{m}+1, N_0\} \cdot T_n = \mathcal{O}\left( \max\{\hat{m}+1, N_0\} \cdot \frac{M^2}{ \mu^{\frac{2}{p}}} \cdot \left( \sum_{k=0}^{m-1} \frac{\epsilon_{k+1}}{\epsilon_{k}^{3-\frac{2}{p}} } + \sum_{k=0}^{m-1} \frac{\epsilon_{k+1}^{\frac{2}{p}} }{\epsilon_{k}^2} \right) \right) .
\end{equation}
For $1<p<2$, if one sets $\epsilon_k = \frac{\epsilon}{2e}\cdot \exp(c^k)$ for any $c\in (1, p)$, then $\hat{m}= \Theta\left( \log\log\left( \frac{\Delta_0}{\epsilon} \right) \right)$, $N_0 = \Omega\left( \log\log\left( \frac{\Delta_0}{\epsilon} \right) \right)$, and $\max\{ \hat{m}+1, N_0\} = \mathcal{O}(N_0)$. Consider applying Proposition~\ref{prop:sublinear-rate} to $(a_1,b_1) = (1, 3-\frac{2}{p})$ and $(a_2,b_2) = (\frac{2}{p}, 2)$, which requires $1<c< \frac{b_1}{a_1}= 3-\frac{2}{p}$ and $1<c< \frac{b_2}{a_2}= p$. Here $p\le 3-\frac{2}{p}$ since $1<p<2$. So setting $c\in (1,p)$ satisfies the requirements. By Proposition~\ref{prop:sublinear-rate}, for this choice of $\{\epsilon_k\}$, the iteration complexity and oracle complexity are respectively at most
\begin{equation}
\label{eq:proof-cor-subgrad-non-monotone-eq5}
T_n = \mathcal{O}\left( \frac{M^2}{ \mu^{\frac{2}{p}}} \cdot \frac{1}{\epsilon^{2-\frac{2}{p}}} \right) \quad\text{and}\quad \max\{\hat{m}+1, N_0\}\cdot T_n = \mathcal{O}\left( \frac{M^2}{\mu^{\frac{2}{p}} } \cdot \frac{1}{\epsilon^{2-\frac{2}{p}}} \cdot N_0 \right) ,
\end{equation}
and no other choice of $\{\epsilon_k\}$ can asymptotically improve the two rates in~\eqref{eq:proof-cor-subgrad-non-monotone-eq3} and~\eqref{eq:proof-cor-subgrad-non-monotone-eq4} beyond~\eqref{eq:proof-cor-subgrad-non-monotone-eq5} if $1<p<2$.

Similarly, for $p=1$, if one sets $\epsilon_k = \frac{\epsilon}{2}\cdot c^k$ for any $c>1$, then $\hat{m}= \Theta\left( \log\left( \frac{\Delta_0}{\epsilon} \right) \right)$, $N_0 = \Omega\left( \log\left( \frac{\Delta_0}{\epsilon} \right) \right)$, and $\max\{ \hat{m}+1, N_0\} = \mathcal{O}(N_0)$. Then by Proposition~\ref{prop:linear-rate}, the iteration complexity and oracle complexity are respectively at most
\begin{equation}
\label{eq:proof-cor-subgrad-non-monotone-eq6}
T_n= \mathcal{O}\left( \frac{M^2}{ \mu^{\frac{2}{p}}} \cdot \log\left( \frac{\Delta_0}{\epsilon} \right) \right) \quad\text{and}\quad \max\{ \hat{m}+1, N_0\} \cdot T_n = \mathcal{O}\left( \frac{M^2}{ \mu^{\frac{2}{p}}} \cdot \log\left( \frac{\Delta_0}{\epsilon} \right) \cdot N_0 \right) ,
\end{equation}
and no other choice of $\{\epsilon_k\}$ can asymptotically improve the two rates in~\eqref{eq:proof-cor-subgrad-non-monotone-eq3} and~\eqref{eq:proof-cor-subgrad-non-monotone-eq4} beyond~\eqref{eq:proof-cor-subgrad-non-monotone-eq6} if $p=1$.
\end{proof}

%% file: appendix.tex
\section{Derivations for Decrement Guarantees}
\label{Sect:fom-decrement}

In this section, we derive the explicit forms of $K_{\mathtt{dec}}(\epsilon, \delta, D)$ for the three examples considered in Section~\ref{Sect:cor-monotone}.

\begin{proposition}[Accelerated method]
\label{prop:accel-method}
Suppose $f$ is convex and $L$-smooth. For the accelerated method in~\eqref{eq:accel}, Assumption~\ref{assumption:decrement-for-fom} holds with
\begin{equation*}
K_{\mathtt{dec}}(\epsilon, \delta, D) = \sqrt{\frac{2LD^2}{\delta-\epsilon}}.
\end{equation*}
\end{proposition}

\begin{proof}
We have the standard result (see~\cite[Theorem 4.4]{Beck2009fista} for example), $f(x^{(n)})-f_\star \le \frac{2L \operatorname{dist}(x^{(0)}, X_\star)^2}{(n+1)^2}$, which implies that a decrement of $\epsilon$ is reached within $\left\lfloor \sqrt{ \frac{2LD^2}{\delta-\epsilon} } \right\rfloor$ iterations.
\end{proof}

\begin{proposition}[Smoothing method]
\label{prop:smooth-method}
Suppose an $(\alpha,\beta)$-convex smoothing family of $f$ is known. For the smoothing method discussed in Section~\ref{Subsec:smooth-monotone}, Assumption~\ref{assumption:decrement-for-fom} holds with
\begin{equation*}
K_{\mathtt{dec}}(\epsilon, \delta, D) = \sqrt{\frac{16\alpha\beta D^2}{\epsilon(\delta-\frac{3}{2}\epsilon)}} .
\end{equation*}
\end{proposition}

\begin{proof}[Proof of Proposition~\ref{prop:smooth-method}]
As in Assumption~\ref{assumption:decrement-for-fom}, we consider $f(x^{(0)}) - f_\star \ge \delta \ge 2\epsilon >0$ and $\operatorname{dist}(x^{(0)}, X_\star) \le D$, where $f_\star$ and $X_\star$ are the minimum value and the set of minimizers of $f$ as defined throughout the paper.

Let $\eta = \frac{\epsilon}{4\beta}$ and $f_{\eta,\star} = \min_{x} f_\eta(x)$, where $f_\eta$ is given by the $(\alpha,\beta)$-convex smoothing family. Denote the sublevel set $X_{\eta}(t) := \{x: f_\eta(x)\le f_{\eta,\star} +t\}$. Recall that 
\begin{equation}
\label{eq:proof-prop-smooth-method-eq1}
f(x)\le f_\eta(x) \le f(x)+\beta\eta = f(x)+\frac{\epsilon}{4}, \quad\forall x.
\end{equation}
Then we have $f_\eta(x^{(0)}) \ge f(x^{(0)})$ and $f_{\eta,\star} \le f_\eta(x_\star) \le f(x_\star)+ \frac{\epsilon}{4}$ for any $x_\star \in X_\star$. Combining these two inequalities gives
\begin{equation}
\label{eq:proof-prop-smooth-method-eq2}
f_\eta(x^{(0)}) - f_{\eta,\star} \ge f(x^{(0)}) - (f_\star + \frac{\epsilon}{4}) \ge \delta - \frac{\epsilon}{4}.
\end{equation}
Consider applying the accelerated method (initialized at $x^{(0)}$) to the $(\frac{\alpha}{\eta})$-smooth convex function $f_\eta$. Following a modified analysis of~\cite[Theorem 4.4]{Beck2009fista} (see~\cite[Appendix B]{Renegar2022restart}), for any $\hat{\epsilon} >0$, this method returns an iterate $x^{(k)}$ satisfying $f_\eta(x^{(k)}) \le f_{\eta,\star} + \hat{\epsilon}$ with iterations at most
\begin{equation}
\label{eq:proof-prop-smooth-method-eq3}
k\le 2\operatorname{dist}(x^{(0)}, X_{\eta}(\hat{\epsilon}/2)) \sqrt{\frac{\alpha}{\eta}\cdot \frac{1}{\hat{\epsilon}}} .
\end{equation}
Let $\hat{\epsilon} = f_\eta(x^{(0)}) - f_{\eta,\star} - \frac{5}{4}\epsilon$. From~\eqref{eq:proof-prop-smooth-method-eq2}, we have
\begin{equation}
\label{eq:proof-prop-smooth-method-eq4}
\hat{\epsilon}\ge (\delta- \frac{1}{4}\epsilon) - \frac{5}{4}\epsilon = \delta- \frac{3}{2}\epsilon \ge \frac{1}{2}\epsilon >0.
\end{equation}
We fix this specific choice of $\hat{\epsilon}$ in the remainder of the proof. Then~\eqref{eq:proof-prop-smooth-method-eq3} implies $f_\eta(x^{(k)}) \le f_\eta(x^{(0)})- \frac{5}{4}\epsilon$ with iterations at most
\begin{equation}
\label{eq:proof-prop-smooth-method-eq5}
k\le 2\operatorname{dist}(x^{(0)}, X_{\eta}(\hat{\epsilon}/2)) \sqrt{\frac{\alpha}{\eta} \cdot \frac{1}{f_\eta(x^{(0)}) - f_{\eta,\star} - \frac{5}{4}\epsilon}} \le 2\operatorname{dist}(x^{(0)}, X_{\eta}(\hat{\epsilon}/2)) \sqrt{\frac{4\alpha\beta}{\epsilon (\delta - \frac{3}{2}\epsilon)}} ,
\end{equation}
where the last step is by~\eqref{eq:proof-prop-smooth-method-eq2}.
By the properties in~\eqref{eq:proof-prop-smooth-method-eq1}, we have $f(x^{(0)}) - f(x^{(k)}) \ge (f_\eta(x^{(0)}) - \frac{1}{4}\epsilon) - f_\eta(x^{(k)})$, which further implies
\begin{equation}
\label{eq:proof-prop-smooth-method-eq6}
f_\eta(x^{(k)}) \le f_\eta(x^{(0)})- \frac{5}{4}\epsilon \quad\implies\quad f(x^{(k)}) \le f(x^{(0)}) - \epsilon.
\end{equation}
Consider an arbitrary fixed point $x_\star\in X_\star$. Using the properties in~\eqref{eq:proof-prop-smooth-method-eq1} and the optimality of $x_\star$, we obtain
\begin{equation*}
f_\eta(x_\star) \le f(x_\star) + \frac{\epsilon}{4} \le f(x)+ \frac{\epsilon}{4} \le f_\eta(x)+ \frac{\epsilon}{4}, \quad\forall x.
\end{equation*}
By taking the infimum over $x$ and using~\eqref{eq:proof-prop-smooth-method-eq4}, we have that $f_\eta(x_\star) \le f_{\eta,\star}+ \frac{\epsilon}{4} \le f_{\eta,\star}+ \frac{\hat{\epsilon}}{2}$, which implies $x_\star \in X_\eta(\hat{\epsilon}/2)$. This further implies $X_\star \subseteq X_\eta(\hat{\epsilon}/2)$ since $x_\star$ is arbitrary. So
\begin{equation}
\label{eq:proof-prop-smooth-method-eq7}
\operatorname{dist}(x^{(0)}, X_{\eta}(\hat{\epsilon}/2)) \le \operatorname{dist}(x^{(0)}, X_{\star}) . 
\end{equation}
Combining~\eqref{eq:proof-prop-smooth-method-eq6} and~\eqref{eq:proof-prop-smooth-method-eq7} with the statement in~\eqref{eq:proof-prop-smooth-method-eq5}, we obtain the desired result: the aforementioned method returns an iterate $x^{(k)}$ satisfying $f(x^{(k)}) \le f(x^{(0)}) - \epsilon$ with iterations at most
\begin{equation*}
k\le 2 \operatorname{dist}(x^{(0)}, X_{\star}) \sqrt{\frac{4\alpha\beta}{\epsilon (\delta - \frac{3}{2}\epsilon)}} \le \sqrt{\frac{16\alpha\beta D^2}{\epsilon(\delta-\frac{3}{2}\epsilon)}}.
\end{equation*}
\end{proof}

\begin{proposition}[Subgradient method]
\label{prop:subgrad-method}
Suppose $f$ is convex and $M$-Lipschitz. 
For the subgradient method discussed in Section~\ref{Subsec:subgrad-monotone}, Assumption~\ref{assumption:decrement-for-fom} holds with
\begin{equation*}
K_{\mathtt{dec}}(\epsilon, \delta, D) = \frac{M^2 D^2}{2\epsilon (\delta -\frac{3}{2}\epsilon)}.
\end{equation*}
\end{proposition}

\begin{proof}[Proof of Proposition~\ref{prop:subgrad-method}]
As in Assumption~\ref{assumption:decrement-for-fom}, we consider $f(x^{(0)}) - f_\star \ge \delta \ge 2\epsilon >0$ and $\operatorname{dist}(x^{(0)}, X_\star) \le D$. 
Consider an arbitrary point $x_\star \in X_\star$. If $g^{(i)} \ne 0$, then $x^{(i+1)} = x^{(i)} - \frac{\epsilon}{\|g^{(i)}\|_2^2} g^{(i)}$, which implies
\begin{equation*}
\begin{split}
\|x^{(i+1)} - x_\star\|_2^2 &= \left\|x^{(i)} - x_\star - \frac{\epsilon}{\|g^{(i)}\|_2^2} g^{(i)} \right\|_2^2 = \|x^{(i)} - x_\star\|_2^2 + \frac{2\epsilon}{\|g^{(i)}\|_2^2} \langle x_\star - x^{(i)}, g^{(i)} \rangle + \frac{\epsilon^2}{\|g^{(i)}\|_2^2} \\
&\le \|x^{(i)} - x_\star\|_2^2 + \frac{2\epsilon}{\|g^{(i)}\|_2^2} (f(x_\star) - f(x^{(i)})) + \frac{\epsilon^2}{\|g^{(i)}\|_2^2} \\
&= \|x^{(i)} - x_\star\|_2^2 - \frac{2\epsilon}{\|g^{(i)}\|_2^2} \left( f(x^{(i)}) - f_\star - \frac{\epsilon}{2} \right) .
\end{split}
\end{equation*}
Suppose $g^{(i)} \ne 0$ and $f(x^{(i)})> f(x^{(0)})-\epsilon$ for all $i=0,\dots, n$ for some $n$. Then $f(x^{(i)}) - f_\star - \frac{\epsilon}{2} > f(x^{(0)}) - f_\star - \frac{3}{2}\epsilon \ge \delta- \frac{3}{2}\epsilon >0$. 
The following holds for all $0\le i\le n$.
\begin{equation*}
\|x^{(i+1)} - x_\star\|_2^2 \le \|x^{(i)} - x_\star\|_2^2 - \frac{2\epsilon}{\|g^{(i)}\|_2^2} \left( f(x^{(i)}) - f_\star - \frac{\epsilon}{2} \right) \le \|x^{(i)} - x_\star\|_2^2 - \frac{2\epsilon}{M^2} \left( \delta- \frac{3}{2}\epsilon \right).
\end{equation*}
Summing over $i = 0,\dots, n$ gives
\begin{equation*}
\|x^{(0)} - x_\star\|_2^2 \ge \sum_{i=0}^{n} \left( \|x^{(i)} - x_\star\|_2^2 - \|x^{(i+1)} - x_\star\|_2^2 \right) \ge \frac{2\epsilon (n+1)}{M^2} \left( \delta- \frac{3}{2}\epsilon \right).
\end{equation*}
So $n+1 \le \frac{M^2 \|x^{(0)} - x_\star\|_2^2}{2\epsilon (\delta -\frac{3}{2}\epsilon)} $.
Since this holds for any $x_\star \in X_\star$, we have $n+1 \le \frac{M^2 \operatorname{dist}(x^{(0)}, X_\star)^2}{2\epsilon (\delta -\frac{3}{2}\epsilon)} \le \frac{M^2 D^2}{2\epsilon (\delta -\frac{3}{2}\epsilon)} $. Therefore, 
\begin{equation*}
\min\left\{ f(x^{(i)}): 0\le i\le \frac{M^2 D^2}{2\epsilon (\delta -\frac{3}{2}\epsilon)} \right\} \le f(x^{(0)}) -\epsilon. \qedhere
\end{equation*}
\end{proof}

\section{Deferred Proofs}
\label{Sect:technical-propositions}

In this section, we prove the linear rate in Proposition~\ref{prop:linear-rate} and the sublinear rate in Proposition~\ref{prop:sublinear-rate}. For each of these propositions, we show the upper-bound attainability and the matching lower bounds separately.

\subsection{Proof of Proposition~\ref{prop:linear-rate}}

We divide the proof into the following two propositions. Proposition~\ref{prop:attainable-linear-rate} verifies that the choice $\epsilon_k = \frac{\epsilon}{2}\cdot c_0^k$ attains the two claimed rates in Proposition~\ref{prop:linear-rate}. Proposition~\ref{prop:lowerbound-linear-rate} establishes the lower bounds for the two minimization problems in Proposition~\ref{prop:linear-rate}. Note $m\le \hat{m}$ from~\eqref{eq:m}, so Proposition~\ref{prop:lowerbound-linear-rate} gives a valid lower bound despite a slight difference in the statement.

\begin{proposition}
\label{prop:attainable-linear-rate}
Consider any fixed $a>0$ and $0<\epsilon< \Delta_0$. Set $\epsilon_k = \frac{\epsilon}{2}\cdot c_0^k$ for any $c_0>1$. If $m$ and $\hat{m}$ are defined by~\eqref{eq:m}, then
\begin{equation*}
\sum_{k=0}^{m-1} \frac{\epsilon_{k+1}^a}{\epsilon_{k}^a} = \mathcal{O}\left( \log \left(\frac{\Delta_0}{\epsilon} \right) \right) , \quad\text{and}\quad (\hat{m}+1) \cdot \sum_{k=0}^{m-1} \frac{\epsilon_{k+1}^a}{\epsilon_{k}^a} = \mathcal{O}\left( \left( \log \left(\frac{\Delta_0}{\epsilon} \right) \right)^2 \right) .
\end{equation*}
\end{proposition}

\begin{proof}[Proof of Proposition~\ref{prop:attainable-linear-rate}]
For any $k\ge 0$,
\begin{equation}
\label{eq:proof-prop-attainable-linear-rate-eq1}
\frac{\epsilon_{k+1}^a}{\epsilon_{k}^a} = \frac{(\frac{\epsilon}{2})^a \cdot c_0^{a(k+1)}}{ (\frac{\epsilon}{2})^a \cdot c_0^{ak}} = c_0^a . 
\end{equation}
We also need to control $m$ and $\hat{m}$. By~\eqref{eq:m}, $m= \inf\{k: \epsilon_k \ge \Delta_0/2 \}$ and $\hat{m}= \inf\{k: \epsilon_k > \Delta_0 \}$, so
\begin{equation}
\label{eq:proof-prop-attainable-linear-rate-eq2}
m = \inf \left\{k: c_0^k \cdot \frac{\epsilon}{2} \ge \frac{\Delta_0}{2} \right\} = \left\lceil \frac{\log (\frac{\Delta_0}{\epsilon})}{\log c_0}\right\rceil = \mathcal{O}\left( \log \left(\frac{\Delta_0}{\epsilon} \right) \right) ,
\end{equation}
\begin{equation}
\label{eq:proof-prop-attainable-linear-rate-eq3}
\hat{m} = \inf \left\{k: c_0^k \cdot \frac{\epsilon}{2} > \Delta_0 \right\} \le \frac{\log (\frac{2\Delta_0}{\epsilon})}{\log c_0}+1 = \mathcal{O}\left( \log \left(\frac{\Delta_0}{\epsilon} \right) \right) .
\end{equation}
Combining~\eqref{eq:proof-prop-attainable-linear-rate-eq1}--\eqref{eq:proof-prop-attainable-linear-rate-eq3} gives $\sum_{k=0}^{m-1} \frac{\epsilon_{k+1}^a}{\epsilon_{k}^a} = c_0^a \cdot m = \mathcal{O}\left( \log \left(\frac{\Delta_0}{\epsilon} \right) \right)$, and $(\hat{m}+1) \sum_{k=0}^{m-1} \frac{\epsilon_{k+1}^a}{\epsilon_{k}^a} = c_0^a \cdot m (\hat{m}+1) = \mathcal{O}\left( \left( \log \left(\frac{\Delta_0}{\epsilon} \right) \right)^2 \right)$.
\end{proof}

\begin{proposition}
\label{prop:lowerbound-linear-rate}
Consider any $a>0$ and $0<\epsilon< \Delta_0$. The following holds for any integer $m\ge 1$ and any positive $\epsilon_0,\dots,\epsilon_m$ satisfying $\epsilon_0\le \frac{\epsilon}{2}$ and $\epsilon_m \ge \frac{\Delta_0}{2}$.
\begin{equation*}
\sum_{k=0}^{m-1} \frac{\epsilon_{k+1}^a}{\epsilon_{k}^a} \ge (ea) \log\left( \frac{\Delta_0}{\epsilon} \right), \quad\text{and}\quad m \sum_{k=0}^{m-1} \frac{\epsilon_{k+1}^a}{\epsilon_{k}^a} \ge \frac{e^2 a^2}{4} \left( \log \left( \frac{\Delta_0}{\epsilon} \right) \right)^2 .
\end{equation*}
\end{proposition}

\begin{proof}[Proof of Proposition~\ref{prop:lowerbound-linear-rate}]
For any fixed integer $m\ge 1$, noting that $\epsilon_0\le \frac{\epsilon}{2}$ and $\epsilon_m \ge \frac{\Delta_0}{2}$, we have
\begin{equation}
\label{eq:proof-prop-lowerbound-linear-rate-eq1}
\sum_{k=0}^{m-1} \frac{\epsilon_{k+1}^a}{\epsilon_{k}^a} \ge m \left( \prod_{k=0}^{m-1} \frac{\epsilon_{k+1}^a}{\epsilon_{k}^a} \right)^{\frac{1}{m}} = m \left( \frac{\epsilon_{m}}{\epsilon_{0}} \right)^{\frac{a}{m}} \ge m \left( \frac{\Delta_0}{\epsilon} \right)^{\frac{a}{m}}.
\end{equation}
Consider the two functions $F_1(y) = \log y + \frac{a}{y}\log (\frac{\Delta_0}{\epsilon})$ and $F_2(y) = 2\log y + \frac{a}{y}\log (\frac{\Delta_0}{\epsilon})$, which are both defined for real numbers $y>0$. Since $F'_1(y) = \frac{1}{y} - \frac{a}{y^2} \log (\frac{\Delta_0}{\epsilon})$ and $F'_2(y) = \frac{2}{y} - \frac{a}{y^2} \log (\frac{\Delta_0}{\epsilon})$, the following results hold for any real $y>0$.
\begin{equation}
\label{eq:proof-prop-lowerbound-linear-rate-eq2}
F_1(y) \ge F_1\left( a\log \left( \frac{\Delta_0}{\epsilon} \right) \right) = \log \left( a\log \left( \frac{\Delta_0}{\epsilon} \right) \right) + 1 ,
\end{equation}
\begin{equation}
\label{eq:proof-prop-lowerbound-linear-rate-eq3}
F_2(y) \ge F_2\left( \frac{a}{2} \log \left( \frac{\Delta_0}{\epsilon} \right) \right) = 2\log \left( \frac{a}{2} \log \left( \frac{\Delta_0}{\epsilon} \right) \right) + 2.
\end{equation}
Now for arbitrary integer $m\ge 1$, using~\eqref{eq:proof-prop-lowerbound-linear-rate-eq1} and~\eqref{eq:proof-prop-lowerbound-linear-rate-eq2}, we obtain
\begin{equation*}
\sum_{k=0}^{m-1} \frac{\epsilon_{k+1}^a}{\epsilon_{k}^a} \ge m \left( \frac{\Delta_0}{\epsilon} \right)^{\frac{a}{m}} = \exp (F_1(m)) \ge e\cdot a\log \left( \frac{\Delta_0}{\epsilon} \right).
\end{equation*}
Similarly, using~\eqref{eq:proof-prop-lowerbound-linear-rate-eq1} and~\eqref{eq:proof-prop-lowerbound-linear-rate-eq3} gives
\begin{equation*}
m \sum_{k=0}^{m-1} \frac{\epsilon_{k+1}^a}{\epsilon_{k}^a} \ge m^2 \left( \frac{\Delta_0}{\epsilon} \right)^{\frac{a}{m}} = \exp (F_2(m)) \ge e^2\cdot \left( \frac{a}{2}\log \left( \frac{\Delta_0}{\epsilon} \right) \right)^2 . \qedhere
\end{equation*}
\end{proof}

\subsection{Proof of Proposition~\ref{prop:sublinear-rate}}
We prove this via two main propositions, handling the attainability and lower bounds, respectively. Proposition~\ref{prop:attainable} verifies that the choice of the form $\epsilon_k = \frac{\epsilon}{2e}\cdot \exp(c_0^k)$ attains the two claimed rates in Proposition~\ref{prop:sublinear-rate}.
For the lower bounds in Proposition~\ref{prop:sublinear-rate}, the first one is immediate, since $\sum_{k=0}^{m-1} \frac{\epsilon_{k+1}^a}{\epsilon_{k}^b} \ge \frac{\epsilon_{1}^a}{\epsilon_{0}^b} \ge \frac{\epsilon_{0}^a}{\epsilon_{0}^b} = \Omega(\frac{1}{\epsilon^{b-a}})$ where the last step is because $\epsilon_0\le \frac{\epsilon}{2}$. The second lower bound in Proposition~\ref{prop:sublinear-rate} can be given by Proposition~\ref{prop:lowerbound}, noting that $m\le \hat{m}$ from~\eqref{eq:m}. The proof of Proposition~\ref{prop:lowerbound} requires some technical lemmas, given in Appendix~\ref{Subsec:technical-lemmas}.

\begin{proposition}
\label{prop:attainable}
Consider any fixed $0<a<b$, $0<\epsilon< \Delta_0$ with a sufficiently large ratio $\frac{\Delta_0}{\epsilon}$. Set $\epsilon_k = \frac{\epsilon}{2e}\cdot \exp(c_0^k)$ for any $c_0\in (1, \frac{b}{a})$. If $m$ and $\hat{m}$ are defined by~\eqref{eq:m}, then
\begin{equation*}
\sum_{k=0}^{m-1} \frac{\epsilon_{k+1}^a}{\epsilon_{k}^b} = \mathcal{O}\left( \frac{1}{\epsilon^{b-a}} \right) , \quad\text{and}\quad (\hat{m}+1) \cdot \sum_{k=0}^{m-1} \frac{\epsilon_{k+1}^a}{\epsilon_{k}^b} = \mathcal{O}\left( \frac{1}{\epsilon^{b-a}} \cdot\log\log \left(\frac{\Delta_0}{\epsilon} \right) \right) .
\end{equation*}
\end{proposition}

\begin{proof}[Proof of Proposition~\ref{prop:attainable}]
For any $k\ge 0$,
\begin{equation*}
\frac{\epsilon_{k+1}^a}{\epsilon_{k}^b} = \frac{(\frac{\epsilon}{2e})^a \cdot \exp(a\cdot c_0^{k+1})}{ (\frac{\epsilon}{2e})^b \cdot \exp(b\cdot c_0^{k})} = \left( \frac{2e}{\epsilon} \right)^{b-a} \cdot \exp\left( -(b- a\cdot c_0) c_0^k \right) . 
\end{equation*}
Note that $c_0>1$ and $b- a\cdot c_0 >0$, so $\sum_{k=0}^{\infty} \exp\left( -(b- a\cdot c_0) c_0^k \right)$ is finite, which further implies $\sum_{k=0}^{m-1} \frac{\epsilon_{k+1}^a}{\epsilon_{k}^b} = \mathcal{O}\left( \frac{1}{\epsilon^{b-a}} \right)$. 

We also need to upper bound $\hat{m}$. By~\eqref{eq:m}, $\hat{m}= \inf\{k: \epsilon_k > \Delta_0 \}$, so
\begin{equation*}
\hat{m} = \inf \left\{k: \exp(c_0^k) \cdot \frac{\epsilon}{2e} > \Delta_0 \right\} \le \frac{\log\log (\frac{2e\Delta_0}{\epsilon})}{\log c_0} +1 = \mathcal{O}\left( \log\log \left(\frac{\Delta_0}{\epsilon} \right) \right) .
\end{equation*}
Combining the upper bounds for $\sum_{k=0}^{m-1} \frac{\epsilon_{k+1}^a}{\epsilon_{k}^b}$ and $(\hat{m}+1)$ leads to the desired results.
\end{proof}

\begin{proposition}
\label{prop:lowerbound}
Consider any $0<a<b$ and $0<\epsilon< \Delta_0$. If $\log \left( \frac{\Delta_0}{\epsilon} \right) > \max\{\frac{e^2}{(b-a)^2}, \frac{1}{b-a} \}$, then the following holds for any integer $m\ge 1$ and any positive $\epsilon_0,\dots,\epsilon_m$ satisfying $\epsilon_0\le \frac{\epsilon}{2}$ and $\epsilon_m \ge \frac{\Delta_0}{2}$.
\begin{equation*}
m \sum_{k=0}^{m-1} \frac{\epsilon_{k+1}^a}{\epsilon_{k}^b} \ge \frac{2^{b-a-1}}{\log \frac{b}{a}} \cdot \frac{1}{\epsilon^{b-a}} \cdot \log\log \left( \frac{\Delta_0}{\epsilon} \right) .
\end{equation*}
\end{proposition}

\begin{proof}[Proof of Proposition~\ref{prop:lowerbound}]
By Lemma~\ref{lemma:fixed-m}, for any fixed integer $m\ge 1$, we have
\begin{equation}
\label{eq:proof-prop-lowerbound-eq1}
m \sum_{k=0}^{m-1} \frac{\epsilon_{k+1}^a}{\epsilon_{k}^b} \ge m \cdot c^{\frac{1}{c-1} - \frac{m}{c^m-1}} \cdot \left( \frac{2}{\epsilon} \right)^{b-a} \cdot \left( \frac{\Delta_0}{\epsilon} \right)^{\frac{b-a}{c^m-1}} = c^{\frac{1}{c-1}} \cdot \left( \frac{2}{\epsilon} \right)^{b-a} \cdot \exp(\tilde{L} (c^m)) ,
\end{equation}
where $c=\frac{b}{a}>1$, and the function $\tilde{L}(y)$ is defined for any real number $y>1$ as follows
\begin{equation*}
\tilde{L}(y) := \log \left( \frac{\log y}{\log c} \right) - \frac{1}{y-1}\log y + \frac{b-a}{y-1} \cdot \log \left( \frac{\Delta_0}{\epsilon} \right) .
\end{equation*}
By Lemma~\ref{lemma:loglog-bound}, $\exp(\tilde{L}(y)+ \log\log c) = \Omega\left( \log\log \left( \frac{\Delta_0}{\epsilon} \right) \right)$ for any $y>1$, and more precisely, the following holds if $\log \left( \frac{\Delta_0}{\epsilon} \right) > \max\{\frac{e^2}{(b-a)^2}, \frac{1}{b-a} \}$, 
\begin{equation}
\label{eq:proof-prop-lowerbound-eq2}
(\log c) \inf_{y>1} \{\exp(\tilde{L}(y))\}= \inf_{y>1} \{\exp(\tilde{L}(y)+ \log\log c)\} \ge \frac{1}{2} \log\log \left( \frac{\Delta_0}{\epsilon} \right) .
\end{equation}
Now for arbitrary $m$, combining~\eqref{eq:proof-prop-lowerbound-eq1} and~\eqref{eq:proof-prop-lowerbound-eq2}, if $\log \left( \frac{\Delta_0}{\epsilon} \right) > \max\{\frac{e^2}{(b-a)^2}, \frac{1}{b-a} \}$ then 
\begin{equation*}
\begin{split}
& m \sum_{k=0}^{m-1} \frac{\epsilon_{k+1}^a}{\epsilon_{k}^b} \ge c^{\frac{1}{c-1}} \cdot \left( \frac{2}{\epsilon} \right)^{b-a} \cdot \inf_{m\in \mathbb{N}_+} \{\exp(\tilde{L}(c^m)) \} \\
&\ge \left( \frac{2}{\epsilon} \right)^{b-a} \cdot \inf_{y\in\mathbb{R},y>1} \{\exp(\tilde{L}(y)) \} \ge \left( \frac{2}{\epsilon} \right)^{b-a} \cdot \frac{1}{2\log c} \log\log \left( \frac{\Delta_0}{\epsilon} \right) ,
\end{split}
\end{equation*}
where the second step is because $c>1$. Replacing $c$ by $\frac{b}{a}$ gives the desired result.
\end{proof}

\subsection{Supporting Lemmas for Proposition~\ref{prop:lowerbound}}
\label{Subsec:technical-lemmas}

\begin{lemma}
\label{lemma:loglog-bound}
Consider any $0<a<b$ and $0<\epsilon< \Delta_0$. Define the function
\begin{equation*}
L(y) = \log \log y - \frac{1}{y-1}\log y + \frac{b-a}{y-1} \cdot \log \left( \frac{\Delta_0}{\epsilon} \right) .
\end{equation*}
Then $\exp(L(y)) = \Omega\left( \log\log \left( \frac{\Delta_0}{\epsilon} \right) \right)$ uniformly for any $y>1$, as $\frac{\Delta_0}{\epsilon} \to \infty$. More precisely, if $\log \left( \frac{\Delta_0}{\epsilon} \right) > \max\{\frac{e^2}{(b-a)^2}, \frac{1}{b-a} \}$, then
\begin{equation*}
\inf_{y>1} \{\exp(L(y))\} \ge \frac{1}{2} \log\log \left( \frac{\Delta_0}{\epsilon} \right) .
\end{equation*}
\end{lemma}

\begin{proof}[Proof of Lemma~\ref{lemma:loglog-bound}]
Consider the following three cases. 

(i) If $y\ge \left( \frac{\Delta_0}{\epsilon} \right)^{b-a}$, then
\begin{equation*}
\log\log y > \log\left( (b-a) \log\left( \frac{\Delta_0}{\epsilon} \right) \right) = \log\log \left( \frac{\Delta_0}{\epsilon} \right) + \log(b-a).
\end{equation*}
Also, since $\frac{1}{y-1}\log y \le 1$ and $\frac{b-a}{y-1} \cdot \log \left( \frac{\Delta_0}{\epsilon} \right) \ge 0$, we have
\begin{equation*}
\begin{split}
L(y) &= \log\log y - \frac{1}{y-1}\log y + \frac{b-a}{y-1} \cdot \log \left( \frac{\Delta_0}{\epsilon} \right) \\
&\ge \log\log \left( \left( \frac{\Delta_0}{\epsilon} \right)^{b-a} \right) -1+0 = \log\log \left( \frac{\Delta_0}{\epsilon} \right) + \log(b-a) -1 ,
\end{split}
\end{equation*}
and $\exp (L(y)) = \Omega\left( \log \left( \frac{\Delta_0}{\epsilon} \right) \right) = \Omega\left( \log\log \left( \frac{\Delta_0}{\epsilon} \right) \right)$. In particular, we can obtain the following explicit lower bound given $\log \left( \frac{\Delta_0}{\epsilon} \right) > \frac{e^2}{(b-a)^2}$,
\begin{equation*}
\exp(L(y)) \ge \frac{b-a}{e} \cdot \log\left( \frac{\Delta_0}{\epsilon} \right) \ge \sqrt{\log\left( \frac{\Delta_0}{\epsilon} \right)} \ge \log \sqrt{\log\left( \frac{\Delta_0}{\epsilon} \right)} = \frac{1}{2} \log\log \left( \frac{\Delta_0}{\epsilon} \right) .
\end{equation*}

(ii) If $\left( \frac{\Delta_0}{\epsilon} \right)^{b-a} >y \ge (b-a)\log\left( \frac{\Delta_0}{\epsilon} \right)$ and $y>1$, then
\begin{equation*}
- \frac{1}{y-1}\log y + \frac{b-a}{y-1} \cdot \log \left( \frac{\Delta_0}{\epsilon} \right) \ge 0.
\end{equation*}
Assume $\frac{\Delta_0}{\epsilon}$ is sufficiently large such that $\log\left( \frac{\Delta_0}{\epsilon} \right) > \frac{1}{b-a}$,
then $\log\log y \ge \log\log\left( (b-a)\log\left( \frac{\Delta_0}{\epsilon} \right) \right)$. So we have
\begin{equation}
\label{eq:proof-lemma-loglog-bound-eq1}
L(y) = \log \log y - \frac{1}{y-1}\log y + \frac{b-a}{y-1} \cdot \log \left( \frac{\Delta_0}{\epsilon} \right) \ge \log\log\left( (b-a)\log\left( \frac{\Delta_0}{\epsilon} \right) \right) ,
\end{equation}
and $\exp (L(y)) = \Omega\left( \log\log \left( \frac{\Delta_0}{\epsilon} \right) \right)$. In particular, if $\log \left( \frac{\Delta_0}{\epsilon} \right) > \max\{\frac{1}{(b-a)^2}, \frac{1}{b-a} \}$, then exponentiating~\eqref{eq:proof-lemma-loglog-bound-eq1} gives
\begin{equation}
\label{eq:proof-lemma-loglog-bound-eq2}
\exp (L(y)) \ge \log\log \left( \frac{\Delta_0}{\epsilon} \right) + \log(b-a) \ge \frac{1}{2} \log\log \left( \frac{\Delta_0}{\epsilon} \right) .
\end{equation}

(iii) If $(b-a)\log\left( \frac{\Delta_0}{\epsilon} \right) >y >1$, let $t_0 := \frac{(b-a)}{y} \log\left( \frac{\Delta_0}{\epsilon} \right)$, then $t_0 >1$. Consider the function $h(y,t)$ defined for real numbers $y>1$ and $t>0$ as follows
\begin{equation*}
h(y,t) := \frac{ty - \log y}{y-1} - \frac{\log t}{\log y} .
\end{equation*}
Note that for any $y>1$ and $t\ge 1$, the partial derivative satisfies
\begin{equation*}
\frac{\partial}{\partial t}h(y,t) = \frac{y}{y-1} - \frac{1}{t \log y} \ge \frac{1}{1-\frac{1}{y}} - \frac{1}{\log y} = \frac{1}{\int_{1}^y \frac{1}{s^2} \mathrm{d} s} - \frac{1}{\int_{1}^y \frac{1}{s} \mathrm{d} s} > 0 .
\end{equation*}
So $h(y,t) \ge h(y,1) = \frac{y-\log y}{y-1} >0$ for any $y>1$ and $t\ge 1$. Using our choice of $t_0$, we have
\begin{equation*}
\begin{split}
L(y) &= \log \log y - \frac{1}{y-1}\log y + \frac{b-a}{y-1} \cdot \log \left( \frac{\Delta_0}{\epsilon} \right) = \log\log y + \frac{t_0\cdot y - \log y}{y-1} \\
&\ge \log\log y + \frac{\log t_0}{\log y} \\
&\ge \log(\log y + \log t_0) = \log\log (y\cdot t_0) = \log\log \left( (b-a)\log\left( \frac{\Delta_0}{\epsilon} \right) \right),
\end{split}
\end{equation*}
where the first inequality follows from $h(y, t_0) >0$ and the second inequality is by the concavity of $\log(u)$ at $u=\log y$. This lower bound of $L(y)$ is exactly the same as~\eqref{eq:proof-lemma-loglog-bound-eq1}, so we still have $\exp (L(y)) = \Omega\left( \log\log \left( \frac{\Delta_0}{\epsilon} \right) \right)$ and the same explicit version in~\eqref{eq:proof-lemma-loglog-bound-eq2}.

Combining all three cases, given $\log \left( \frac{\Delta_0}{\epsilon} \right) > \max\{\frac{e^2}{(b-a)^2}, \frac{1}{b-a} \}$, we always have $\inf_{y>1} \{\exp(L(y))\} \ge \frac{1}{2} \log\log \left( \frac{\Delta_0}{\epsilon} \right)$.
\end{proof}

\begin{lemma}
\label{lemma:fixed-m}
For any fixed $0<a<b$, $0<\epsilon< \Delta_0$, and any fixed integer $m\ge 1$, let $v^\star_m$ be the infimum of the following constrained problem
\begin{equation}
\label{eq:lemma-fixed-m-eq1}
v^\star_m = \begin{cases}
\inf & \sum_{k=0}^{m-1} \frac{\epsilon_{k+1}^a}{\epsilon_{k}^b} \\
\rm{subject\ to} & \epsilon_0 \le \frac{\epsilon}{2},\ \epsilon_m \ge \frac{\Delta_0}{2},\ \epsilon_0, \epsilon_1, \cdots, \epsilon_m >0.
\end{cases}
\end{equation}
Then 
\begin{equation*}
v^\star_m \ge c^{\frac{1}{c-1} - \frac{m}{c^m-1}} \cdot \left( \frac{2}{\epsilon} \right)^{b-a} \cdot \left( \frac{\Delta_0}{\epsilon} \right)^{\frac{b-a}{c^m-1}},
\end{equation*}
where $c=\frac{b}{a}$.
\end{lemma}

\begin{proof}[Proof of Lemma~\ref{lemma:fixed-m}]
For any $0\le k\le m-1$, let $y_k = \frac{\epsilon_{k+1}^a}{\epsilon_{k}^{b}}$ and $t_k = a^k\cdot b^{m-1-k}$. Then for any feasible solution of~\eqref{eq:lemma-fixed-m-eq1},
\begin{equation}
\label{eq:proof-lemma-fixed-m-eq1}
\prod_{k=0}^{m-1} y_k^{t_k} = \prod_{k=0}^{m-1} \left( \epsilon_k^{-b\cdot t_k} \cdot \epsilon_{k+1}^{a\cdot t_k} \right) = \epsilon_0^{-b\cdot t_0} \cdot \epsilon_m^{a\cdot t_{m-1}} \cdot \prod_{k=1}^{m-1} \epsilon_k^{a\cdot t_{k-1} - b\cdot t_k} = \epsilon_0^{-b^m} \cdot \epsilon_m^{a^m} \ge \left( \frac{\Delta_0}{2} \right)^{a^m} \cdot \left( \frac{2}{\epsilon} \right)^{b^m}.
\end{equation}
By Lemma~\ref{lemma:Jensen-corollary}, 
\begin{equation}
\label{eq:proof-lemma-fixed-m-eq2}
\sum_{k=0}^{m-1} y_k \ge (t_0+\cdots+t_{m-1}) \left( \frac{\prod_{k=0}^{m-1} y_k^{t_k} }{\prod_{k=0}^{m-1} t_k^{t_k}} \right)^{\frac{1}{t_0+\cdots+t_{m-1}}} .
\end{equation}
Denote $A := \sum_{k=0}^{m-1} t_k$, and $B:= \prod_{k=0}^{m-1} t_k^{t_k}$, which are two constants only depending on $a,b,m$. Combining~\eqref{eq:proof-lemma-fixed-m-eq1} and~\eqref{eq:proof-lemma-fixed-m-eq2} gives
\begin{equation}
\label{eq:proof-lemma-fixed-m-eq3}
\sum_{k=0}^{m-1} \frac{\epsilon_{k+1}^a}{\epsilon_{k}^{b}} = \sum_{k=0}^{m-1} y_k \ge \frac{A}{B^{1/A}} \cdot \left( \left( \frac{\Delta_0}{2} \right)^{a^m} \cdot \left( \frac{2}{\epsilon} \right)^{b^m} \right)^{1/A},
\end{equation}
and the equality is attained if both~\eqref{eq:proof-lemma-fixed-m-eq1} and~\eqref{eq:proof-lemma-fixed-m-eq2} reach equality, which requires $\epsilon_0 = \frac{\epsilon}{2}$, $\epsilon_m = \frac{\Delta_0}{2}$, and $\frac{y_0}{t_0} = \cdots = \frac{y_{m-1}}{t_{m-1}}$ (see Lemma~\ref{lemma:Jensen-corollary}). Note $\frac{y_{k-1}}{y_{k}} = \frac{\epsilon_{k}^a}{\epsilon_{k-1}^{b}}/ \frac{\epsilon_{k+1}^a}{\epsilon_{k}^{b}}$ and $\frac{t_{k-1}}{t_{k}} = \frac{b}{a}$ for any $1\le k\le m-1$. So~\eqref{eq:proof-lemma-fixed-m-eq3} reaches equality if the following holds
\begin{equation}
\label{eq:proof-lemma-fixed-m-eq4}
\begin{cases}
\epsilon_0 = \frac{\epsilon}{2}, \quad \epsilon_m = \frac{\Delta_0}{2}, \\
\frac{\epsilon_{k}^{a}}{\epsilon_{k-1}^b} = \frac{b}{a}\cdot \frac{\epsilon_{k+1}^a}{\epsilon_{k}^{b}}, \quad \forall 1\le k\le m-1 .
\end{cases}
\end{equation}
Consider the following choice of $(\bar{\epsilon}_0, \dots, \bar{\epsilon}_m)$,
\begin{equation*}
\bar{\epsilon}_k = \frac{\epsilon}{2}\cdot \left( \frac{\Delta_0}{\epsilon} \right)^{\frac{c^k-1}{c^m-1}} \cdot c^{\frac{k}{a(c-1)} - \frac{m(c^k-1)}{a(c-1)(c^m-1)}}, \quad \forall 0\le k\le m.
\end{equation*}
Note the following quantity depends on $k$ only through $c^{-k}$,
\begin{equation}
\label{eq:proof-lemma-fixed-m-eq5}
\frac{\bar{\epsilon}_{k+1}^a}{\bar{\epsilon}_{k}^{b}} = \left( \frac{\epsilon}{2} \right)^{a-b} \cdot \left( \frac{\Delta_0}{\epsilon} \right)^{\frac{b-a}{c^m-1}} \cdot c^{-k+ \frac{1}{c-1} - \frac{m}{c^m-1}} .
\end{equation}
By~\eqref{eq:proof-lemma-fixed-m-eq5}, it can be verified that $(\bar{\epsilon}_0, \dots, \bar{\epsilon}_m)$ satisfies~\eqref{eq:proof-lemma-fixed-m-eq4} and is a feasible solution of~\eqref{eq:lemma-fixed-m-eq1}. Then by~\eqref{eq:proof-lemma-fixed-m-eq3}, $v^\star_m$ is attained by $(\bar{\epsilon}_0, \dots, \bar{\epsilon}_m)$, so $v^\star_m = \sum_{k=0}^{m-1} \frac{\bar{\epsilon}_{k+1}^a}{\bar{\epsilon}_{k}^{b}}$. Using~\eqref{eq:proof-lemma-fixed-m-eq5} again, we have
\begin{equation*}
\begin{split}
v^\star_m &= \sum_{k=0}^{m-1} \frac{\bar{\epsilon}_{k+1}^a}{\bar{\epsilon}_{k}^{b}} = \left( \frac{\epsilon}{2} \right)^{a-b} \cdot \left( \frac{\Delta_0}{\epsilon} \right)^{\frac{b-a}{c^m-1}} \cdot c^{\frac{1}{c-1} - \frac{m}{c^m-1}} \sum_{k=0}^{m-1} c^{-k} \\
&\ge \left( \frac{\epsilon}{2} \right)^{a-b} \cdot \left( \frac{\Delta_0}{\epsilon} \right)^{\frac{b-a}{c^m-1}} \cdot c^{\frac{1}{c-1} - \frac{m}{c^m-1}} 
\end{split}
\end{equation*}
which completes the proof.
\end{proof}

\begin{lemma}
\label{lemma:Jensen-corollary}
The following inequality holds for any integer $m\ge 1$ and any sequences of positive real numbers $\{y_k\}_{k=1}^m, \{t_k\}_{k=1}^m$.
\begin{equation*}
\sum_{k=1}^m y_k \ge (t_1+\cdots+t_m) \left( \frac{\prod_{k=1}^m y_k^{t_k} }{\prod_{k=1}^m t_k^{t_k}} \right)^{\frac{1}{t_1+\cdots+t_m}} .
\end{equation*}
Moreover, equality holds if and only if $\frac{y_1}{t_1} = \frac{y_2}{t_2} = \cdots = \frac{y_m}{t_m}$.
\end{lemma}

\begin{proof}[Proof of Lemma~\ref{lemma:Jensen-corollary}]
Since $\log(\cdot)$ is concave, by Jensen's inequality the following holds for any positive $\{x_k\}_{k=1}^m$ and any positive $\{\lambda_k\}_{k=1}^m$ satisfying $\sum_{k=1}^m \lambda_k = 1$.
\begin{equation*}
\sum_{k=1}^m \lambda_k \log(x_k) \le \log \left(\sum_{k=1}^m \lambda_k x_k \right) .
\end{equation*}
Letting $\lambda_k = \frac{t_k}{\sum_{j=1}^m t_j}$ and $x_k = \frac{y_k}{t_k}$ gives
\begin{equation*}
\frac{1}{\sum_{k=1}^m t_k} \sum_{k=1}^m t_k \log\left( \frac{y_k}{t_k} \right) \le \log \left( \frac{\sum_{k=1}^m y_k}{\sum_{k=1}^m t_k} \right) .
\end{equation*}
Taking the exponential of both sides yields
\begin{equation*}
\left( \frac{\prod_{k=1}^m y_k^{t_k} }{\prod_{k=1}^m t_k^{t_k}} \right)^{\frac{1}{t_1+\cdots+t_m}} \le \frac{\sum_{k=1}^m y_k}{\sum_{k=1}^m t_k},
\end{equation*}
which finishes the proof of the desired inequality. The equality condition of the last inequality follows from the previous Jensen's inequality, which requires $\frac{y_1}{t_1} = \cdots = \frac{y_m}{t_m}$ since the function $\log(\cdot)$ is strictly concave and the multipliers $\frac{t_k}{\sum_{j=1}^m t_j}$ are all strictly positive.
\end{proof}